\documentclass[12pt,a4paper,twoside]{amsart}
\usepackage[utf8]{inputenc}
\usepackage{amssymb,enumerate,graphicx,a4wide,hyperref}
\usepackage{dsfont}
\usepackage{mathrsfs}

\newtheorem{thm}{Theorem}

\newtheorem{lemma}{Lemma}
\newtheorem{corollary}{Corollary}
\theoremstyle{definition}
\newtheorem{definition}{Definition}
\newtheorem{remark}{Remark}
\newtheorem{proposition}{Proposition}

\newcommand{\ddbar}{\partial\bar\partial}
\newcommand{\dbar}{\bar\partial}
\newcommand{\1}{\mathds{1}}
\newcommand{\dist}{\operatorname{dist}}
\newcommand{\lipone}{\operatorname{Lip}_{1,1}}
\newcommand{\eps}{\varepsilon}
\newcommand\dotprod[2]{\langle #1 , #2 \rangle}

\newcommand{\Lcal}{L}
\newcommand{\HoL}{H^0 (L)}
\newcommand{\HoLk}{H^0 (L^k)}
\newcommand{\Fcal}{\mathcal F}

\newcommand{\Cbb}{{\mathbb{C}}}
\newcommand{\Rbb}{{\mathbb{R}}}
\newcommand{\CPn}{{\mathbb{CP}^n}}
\newcommand{\Ocal}{{\mathcal{O}}}

\title{Equidistribution estimates for Fekete points on complex manifolds}

\author{Nir Lev}
\address{Centre de Recerca Matem\`atica, Campus de Bellaterra, Edifici C, 08193 Bellaterra (Barcelona), Spain}
\address{Department of Mathematics, Bar-Ilan University, Ramat-Gan 5290002, Israel}
\email{\href{mailto:levnir@math.biu.ac.il}{\texttt{levnir@math.biu.ac.il}}}

\author{Joaquim Ortega-Cerd\`a}
\address{Dept.\ Matem\`atica Aplicada i An\`alisi,
 Universitat  de Barcelona,
Gran Via 585, 08007 Bar\-ce\-lo\-na, Spain}
\email{\href{mailto:jortega@ub.edu}{\texttt{jortega@ub.edu}}}

\thanks{Supported by the project MTM2011-27932-C02-01 and the CIRIT grant
2009SGR-1303}
\date{\today}
\subjclass[2000]{}

\begin{document}
\begin{abstract}
We study the equidistribution of Fekete points in a compact complex manifold.
These are extremal point configurations defined through sections of powers of a
positive line bundle. Their equidistribution is a known result. The novelty of
our approach is that we relate them to the problem of sampling  and
interpolation on line bundles, which allows us to estimate the equidistribution
of the Fekete points quantitatively. In particular we estimate the
Kantorovich-Wasserstein distance of the Fekete points to its limiting measure.
The sampling and interpolation arrays on line bundles are a subject of
independent interest, and we provide necessary density conditions through the
classical approach of Landau, that in this context measures the local dimension
of the space of sections of the line bundle. We obtain a complete geometric
characterization of sampling and interpolation arrays in the case of compact
manifolds of dimension one, and we prove that there are no arrays of both
sampling and interpolation in the more general setting of semipositive line
bundles.
\end{abstract}
\maketitle

\section{Introduction}

\subsection{}
Let $\Lcal$ be a holomorphic line bundle on a compact complex manifold $X$ of
dimension $n$. The space of global holomorphic sections to $\Lcal$ is
denoted by $\HoL$.
If $s_1, \ldots, s_N$ is a basis for $\HoL$ and $x_1,\ldots, x_N$
are $N$ points in $X$, then the Vandermonde-type determinant
\begin{equation*}
\det (s_i (x_j)), \quad 1\leq i,j\leq N,
\end{equation*}
is a section of the pulled-back line bundle $\Lcal^{\boxtimes N}$ over the
manifold $X^N$.
If $\Lcal$ is endowed with a smooth hermitian metric $\phi$, then 
it also induces a natural metric on $\Lcal^{\boxtimes N}$.

A configuration of $N$ points $x_1,\ldots, x_N$ in $X$ is called a \emph{Fekete
configuration} for $(\Lcal, \phi)$ if it maximizes the pointwise norm 
$| \det (s_i (x_j)) |_\phi$. 
It is easy to check that the definition of a Fekete configuration does not
depend on the particular choice of the basis $s_1, \ldots,s_N$ for $\HoL$.
The compactness of $X$ ensures the existence of
Fekete configurations (but in general there need not be a unique one). 

It is interesting to study the distribution of Fekete points with respect to high powers
$L^k$ of the line bundle $\Lcal$, where $L^k$ is endowed with the product metric
$k \phi$. The model example is the complex projective space $X = \CPn$ with
the hyperplane bundle $L = \Ocal(1)$, endowed with the Fubini-Study metric.
The $k$'th power of $L$ is denoted $\Ocal(k)$, and the holomorphic sections
to $\Ocal(k)$ can be identified with the homogeneous polynomials of degree $k$
in $n+1$ variables.
This is in fact the prime example, and it covers in particular the
classical theory of weighted orthogonal multivariate polynomials.

For each $k=1,2,3,\ldots$ let $\Fcal_k$ be a Fekete configuration
for $(\Lcal^k, k \phi)$.
The goal is to provide information on the distribution of Fekete points
$\Fcal_k$ in geometrical terms of the line bundle $(L, \phi)$,
showing that they are ``equidistributed'' on $X$.
We will consider the case when $L$ is an ample
line bundle with a smooth positive metric $\phi$. 
The problem has already been solved by Berman, Boucksom and Witt Nystr{\"o}m
\cite{BerBouWit11} in an even more general context, when $L$ is a big
line bundle with an arbitrary continuous metric on a compact subset $K\subset X$. The
redeeming feature of our approach is that our new proof provides a
quantitative version of the equidistribution.

\begin{thm}
\label{fekete_distribution}
If the line bundle $(L,\phi)$ is positive then
\begin{equation*}
\frac{\# (\Fcal_k \cap B(x,r))}{\# \Fcal_k}=
\Big ( 1+O\Big ( \big (r\sqrt{k}\,\big )^{-1}\Big )\Big ) \,
\frac{\int_{B(x,r)} (i \partial \bar\partial \phi)^n}{\int_X 
(i \partial \bar\partial \phi)^n} 
\end{equation*}
for every $r > 0$, uniformly in $x \in X$.
\end{thm}

Here $\partial \bar\partial \phi$ is the \emph{curvature form} of the
metric $\phi$, which is a globally defined $(1,1)$-form on $X$, 
and $(i \partial \bar\partial \phi)^n$ is the corresponding
volume form on $X$. 
By $B(x,r)$ we mean the ball of radius $r$ centered at the point $x$ in $X$.
To define the balls we endow the manifold $X$ with an arbitrary
hermitian metric, and use the associated distance function.

The result shows, in particular, that the weak limit as $k \to \infty$ of the
probability measures
\begin{equation}\label{Feketemeasure}
\mu_k = \frac{1}{\# \Fcal_k}\sum_{\lambda \in \Fcal_k} \delta_\lambda
\end{equation}
is the measure 
$(i \partial \bar\partial \phi)^n$ divided by its total mass. This is a
special case of the main theorem proved in \cite{BerBouWit11} in the setting
of positive line bundles.

Theorem~\ref{fekete_distribution} provides an even more precise result
quantifying the convergence. It measures the discrepancy between the Fekete 
points and its limit measure. This has been done in the one dimensional 
setting, see for instance \cite[Theorem 3]{RotSer13}. 

We can also estimate the
distance of the Fekete measure $\mu_k$ to its limit 
$\nu$ in the Kantorovich-Wasserstein metric $W$, which 
metrizes the weak convergence of measures (see Section \ref{sec_wasserstein}). 
This is another measure of how close the Fekete points are to its limit.

\begin{thm}\label{wass_estimate}
If the line bundle $(L,\phi)$ is positive then
\[
\frac1{\sqrt k} \lesssim W(\mu_k,\nu)\lesssim \frac1{\sqrt k}
\] as $k \to \infty$.
\end{thm}

Actually, one can see from the proof of Theorem~\ref{wass_estimate} that the 
lower bound $1/{\sqrt k}\lesssim W(\mu_k,\nu)$ holds for any set of points with the 
same cardinality as the Fekete points. Therefore any family of points, no 
matter how evenly distributed along the manifold with respect to $\nu$ are, 
will converge asymptotically at the same rate as the Fekete points. Thus the 
Fekete points are in a sense optimally distributed, as expected.

The scheme that we propose to study this problem is along the line of research
initiated in \cite{MarOrt10} where the Fekete points are related to another
array of points, the sampling and interpolation points. This has been pursued
further in the one-dimensional setting \cite{AmeOrt12} or even in the real
setting of compact Riemannian manifolds, see \cite{OrtPrid12}.

For each $k=1,2,3,\dots$ let $\Lambda_k$ be a finite set of points in $X$.
We assume that $\{\Lambda_k\}$ is a separated array, which
means that the distance between any two distinct points in $\Lambda_k$
is bounded below by a positive constant times $k^{-1/2}$.
We say that $\{\Lambda_k\}$ is a \emph{sampling array} for
$(L, \phi)$ if there are constants $0<A,B<\infty$ such that, for each
large enough $k$ and any section $s\in \HoLk$ we have
\[
A k^{-n} \sum_{\lambda\in \Lambda_k} |s(\lambda)|^2
\leq \int_X \big | s(x)\big|^2 \leq 
B k^{-n} \sum_{\lambda\in \Lambda_k} |s(\lambda)|^2.
\]

We say that $\{\Lambda_k\}$ is an \emph{interpolation array}
for $(L, \phi)$ if there
is a constant $0<C<\infty$ such that, for each large enough $k$ and any
set of values $\{v_\lambda\}_{\lambda\in \Lambda_k}$, where
each $v_\lambda$ is an element of the fiber of $\lambda$ in $\Lcal^k$,
there is a section $s\in \HoLk$ such that 
$s(\lambda)=v_\lambda$ ($\lambda\in \Lambda_k$) and
\[
\int_X \big | s(x) \big |^2 \leq C k^{-n} \sum_{\lambda
\in\Lambda_k} |v_\lambda|^2.
\]

In order to integrate over $X$ in these definitions
we endow $X$ with an arbitrary volume form. It is easy to see that the
definitions of the sampling and interpolation arrays
do not depend on the particular choice of the volume form on $X$.

Our proof of the equidistribution of Fekete points 
(Theorems \ref{fekete_distribution} and  \ref{wass_estimate} above) is 
inspired by the work of Nitzan and Olevskii
\cite{NitOle12}, where they obtain a new proof of a classical result of Landau
on the distribution of sampling and interpolation points in the Paley-Wiener
space. In some sense, Fekete points are ``almost'' sampling and interpolation
points (see Section \ref{secsampintFek} below).

\subsection{}
We believe that the sampling and interpolation arrays on holomorphic line
bundles are a subject of independent interest, so we proceed to a more detailed
study of them. We can use Landau's classical technique \cite{Landau67} to get
necessary geometric conditions for an array of points to be sampling or
interpolation. We could have used also similar techniques to the ones used by
Nitzan and Olevskii \cite{NitOle12} for this purpose.  We have opted rather for
the analysis of Landau's concentration operator, which measures the local
dimension of the sections of the line bundle, to obtain necessary density
conditions for sampling arrays. This approach was suggested earlier by
Berndtsson \cite{Berndtsson03} and Lindholm \cite{Lindholm01} in the context of
holomorphic line bundles. 

Let $\nu_\Lambda^{-} (R)$ (respectively $\nu_\Lambda^{+} (R)$) denote the
infimum (respectively supremum) of the ratio
\begin{equation}\label{equation11}
\frac{k^{-n} \# (\Lambda_k \cap B(x,r))}{\int_{B(x,r)} 
(i \partial \bar\partial \phi)^n}
\end{equation}
over all $x\in X$, and all $k,r$ such that 
$\frac{R}{\sqrt{k}} \leq r \leq\textrm{diam} (X)$.
As before, to define the balls $B(x,r)$ we have fixed an arbitrary
hermitian metric on the manifold $X$.

\begin{thm}
\label{necessary_conditions}
Let the line bundle $(L,\phi)$ be positive,
and $\Lambda=\{\Lambda_k\}$ be a separated array.
\begin{enumerate}
\item[{\rm (i)}] If $\Lambda$ is a sampling array then
\begin{equation*}
\nu_\Lambda^{-} (R) > \frac{1}{\pi^n n!} - O ( R^{-1} ), \quad R 
\to \infty.
\end{equation*}

\item[{\rm (ii)}] If $\Lambda$ is an interpolation array then
\begin{equation*}
\nu_\Lambda^{+} (R) < \frac{1}{\pi^n n!} + O( R^{-1}), 
\quad R \to \infty.
\end{equation*}
\end{enumerate}
\end{thm}

This result yields necessary conditions
in terms of the lower and upper Beurling-Landau densities,
defined by
\[
D^{-} (\Lambda) = \liminf_{R\to\infty} \nu_\Lambda^{-} (R),
\quad \text{and} \quad
D^{+} (\Lambda) = \limsup_{R\to\infty} \nu_\Lambda^{+} (R).
\]

\begin{corollary}\label{densities}
Let the line bundle $(L,\phi)$ be positive 
and $\Lambda=\{\Lambda_k\}$ be a separated array.
If $\Lambda$ is a sampling array then
\begin{equation*}
D^{-} (\Lambda) \geq \frac{1}{\pi^n n!}\,,
\end{equation*}
while if $\Lambda$ is an interpolation array then
\begin{equation*}
D^{+} (\Lambda) \leq \frac{1}{\pi^n n!}\,.
\end{equation*}
\end{corollary}

When the complex manifold $X$ is one-dimensional, i.e.\ we are 
dealing with a compact Riemann surface, we have a more precise result.
In this case there is a complete
geometric characterization of the sampling and interpolation arrays
in terms of the above densities.

\begin{thm}
\label{iff_conditions}
Let $(L,\phi)$ be a positive line bundle over a compact Riemann
surface $X$, and let $\Lambda=\{\Lambda_k\}$ be a separated array.
Then $\Lambda$ is a sampling array if and only if
\begin{equation*}
D^{-} (\Lambda) > \frac{1}{\pi}\,,
\end{equation*}
while it is an interpolation array if and only if
\begin{equation*}
D^{+} (\Lambda) < \frac{1}{\pi}\,.
\end{equation*}
\end{thm}

We remark that the assumption that $\{\Lambda_k\}$ is separated is not
essential, and similar results hold in the general case. This
can be done with standard techniques, see e.g.\ \cite{Marzo07}, so we
will not go into these details in the paper.

\subsection{}
As pointed out in \cite{MarOrt10} Fekete points provide a construction of an
``almost''
sampling and interpolation array, with the critical density.
In particular this shows that the density threshold in Corollary \ref{densities}
is sharp (see Corollary \ref{cor_dens_sharp} in Section \ref{sec_wasserstein}).

In this context a natural question is whether the Fekete points, or possibly
some other array of points, is simultaneously sampling and
interpolation for $(L, \phi)$. In the case when the manifold $X$ is 
one-dimensional, this question is settled in the negative by 
Theorem~\ref{iff_conditions} above. For $n>1$ we do not have strict 
density conditions, and Corollary \ref{densities} does not exclude the existence of
simultaneously sampling and interpolation arrays. Nevertheless,
we will show that such arrays do not exist,
even in the more general setting when the metric $\phi$ is semi-positive and
has at least one point with a strictly positive curvature.

\begin{thm}\label{no_riesz_bases}
Let $L$ be a holomorphic line bundle over a compact projective manifold $X$, and
$\phi$ be a semi-positive smooth hermitian metric on $L$. If there is
a point in $X$ where $\phi$ has a strictly positive curvature, then there
are no arrays which are simultaneously sampling and
interpolation for $(L, \phi)$.
\end{thm}

Here we need to assume that the manifold $X$ is projective.
When the line bundle is positive this is automatically the case, according
to the Kodaira embedding theorem \cite{Kodaira54}.

The non-existence of simultaneously sampling and interpolation 
sequences is a recent result in the classical Bargmann-Fock
space \cite{AscFeiKai11,GroMal11}. To prove Theorem \ref{no_riesz_bases} 
we use the fact that near a point of positive curvature, the 
sections of high powers of the line bundle resemble closely the functions in the
Bargmann-Fock space. Also our proof of Theorem \ref{iff_conditions}
is guided with the same principle.

\subsection{}
The plan of the paper is the following. In Sections~\ref{SecPreliminaries} and
\ref{SecFeketeProperties} we 
provide the basic properties of the Fekete points, and of the Hilbert space of
holomorphic sections that will be the main tool to study them. In
Section~\ref{secsampintFek} we introduce the sampling and interpolation arrays
and discuss their relationship with the Fekete points. In Section~\ref{Landconc}
we study Landau's concentration operator, that will allow us to
measure the local dimension of the space of sections 
essentially concentrated in a given ball, and use this local dimension to
estimate the number of points in an interpolation or sampling array. In
Section~\ref{curvdens} we estimate the density of the interpolation
and sampling arrays in terms of the volume form associated
to the curvature of the line bundle. In Section~\ref{sec_wasserstein}
we give an  estimate from above and below on the number of Fekete points that lie in a given ball.
We also provide the upper and lower bounds for the Kantorovich-Wasserstein distance between 
the Fekete measure \eqref{Feketemeasure} and its limiting measure.

Next we proceed to a more detailed study of the sampling and 
interpolation arrays. In Section~\ref{simultaneous} we prove
that in a big line bundle with a semipositive metric, whenever there is a point
of positive curvature there are no arrays that are simultaneously sampling
and interpolation. Finally in Section~\ref{onedim} we obtain a
geometric characterization of sampling and interpolation arrays for positive
line bundles over compact manifolds of dimension one.

\subsection*{Acknowledgement}

Part of this work was done while Nir Lev was staying at the
Centre de Recerca Matem\`atica (CRM) in Barcelona, and he would like
to express his gratitude to the institute for the hospitality
and support during his stay.


\section{Preliminaries}\label{SecPreliminaries}

In this section we recall some basic properties of holomorphic line bundles
over complex manifolds. For these and other elementary facts on this subject,
stated below without proofs, the reader may consult \cite{Berndtsson10}.

\subsection{Line bundles}
Below $X$ will be a compact complex manifold of dimension $n$, endowed
with a smooth hermitian metric $\omega$. The metric $\omega$
induces a distance function $d(x,y)$ on $X$, which will be used to define 
the balls $B(x,r) = \{y \in X : d(x,y) < r\}$. 
The hermitian metric $\omega$ also induced a volume form $V$ on $X$, 
which will be used to integrate over $X$.
We emphasize that the choice of the metric $\omega$ is arbitrary,
and the results will not depend on the particular choice made.

By $\Lcal$ we denote a holomorphic line bundle over the manifold $X$.
We assume that $\Lcal$ is endowed with a smooth hermitian metric 
$\phi$, which is a smoothly varying norm on each fiber. It has to
be understood as a collection of functions $\phi_i$ defined on 
trivializing open sets $U_i$ which cover $X$, and satisfying the
compatibility conditions
\[
 \phi_i-\phi_j =\log|g_{ij}|^2,
\]
where $g_{ij}$ are the transition functions of the line bundle $L$ on $U_i\cap U_j$.
If $s$ is a section to $L$ represented by a collection of local functions $s_i$ 
such that $s_i = g_{ij} s_j$, then
\[
|s(x)|^2 = |s_i(x)|^2 e^{-\phi_i(x)}.
\]
We also have an associated scalar product, defined in a similar way by
\[
\dotprod{u(x)}{v(x)} = u_i(x) \overline{v_i(x)} e^{-\phi_i(x)}.
\]

If $\phi$ is the hermitian metric on $L$, then
$\partial \bar\partial \phi$ is a globally defined $(1,1)$-form on $X$, which is
called the \emph{curvature form} of the metric $\phi$.
The line bundle $L$ with the metric $\phi$ is called \emph{positive} if
$i \partial \bar\partial \phi$ is a positive form. Equivalently, $L$ with the
metric $\phi$ is positive if the representative of $\phi$ 
with respect to any local trivialization is a strictly plurisubharmonic
function.
We remark that in the case when $\phi$ is positive, the curvature form
$\partial \bar\partial \phi$ may be used to define a natural metric 
on $X$, which in turn induces a distance function and a volume form on $X$.
However, we find it convenient to work with an arbitrary metric $\omega$,
which is not necessarily related to the curvature form.

We will use the notation $\lesssim$ to indicate an implicit multiplicative
constant which may depend only on the hermitian manifold $(X, \omega)$ and the hermitian 
line bundle $(L, \phi)$.

The space of global holomorphic sections to $\Lcal$ will be denoted $\HoL$.
This is a finite-dimensional space, satisfying the estimate
\[
\dim \HoLk \lesssim k^n.
\]
While the latter estimate holds for an arbitrary line bundle on a compact manifold,
in the case when the line bundle $L$ is big there is also a similar estimate from
below, i.e.\
\begin{equation}\label{h_dim_pos}
k^n \lesssim \dim \HoLk \lesssim k^n.
\end{equation}
In particular this holds whenever the line bundle $L$ is positive.

If $L$ is a line bundle over $X$ and $M$ is a line bundle over $Y$, we
denote by $L\boxtimes M$ the line bundle over the product manifold
$X\times Y$ defined as $L \boxtimes M= \pi_X^*(L)\otimes \pi_Y^*(M)$,
where $\pi_X: X\times Y\to X$ is the projection onto the first factor
and $\pi_Y: X\times Y\to Y$ is the projection onto the second.

\subsection{Bergman kernel}
The space $\HoL$ admits a Hilbert space structure when endowed with the scalar product
\begin{equation*}
\dotprod{u}{v} = \int_X \dotprod{u(x)}{v(x)}, \quad u,v \in \HoL,
\end{equation*}
where the integration is taken with respect to the volume form $V$.

The Bergman kernel $\Pi (x,y)$ associated to this space is 
a section to the line bundle $\Lcal \boxtimes \bar\Lcal$
over the manifold $X\times X$, defined by
\begin{equation}\label{EqBergmanKernelBasis}
\Pi (x,y)=\sum_{j=1}^N s_j (x) \otimes \overline{s_j (y)},
\end{equation}
where $s_1,\ldots, s_N$ is an orthonormal basis for $\HoL$.
It is easy to check that this definition does not depend on the particular choice
of the orthonormal basis $s_1,\ldots, s_N$.
The Bergman kernel $\Pi (x,y)$ is in a sense the reproducing kernel 
for the space $\HoL$, satisfying the reproducing formula
\begin{equation*}
s(x) = \int_X \big \langle s(y), \Pi (x,y)\big \rangle \, dV(y)
\end{equation*}
for $s \in \HoL$.
The pointwise norm of the Bergman kernel is symmetric,
\begin{equation}\label{EqBergmanSymm}
| \Pi(x,y) | = | \Pi(y,x) |.
\end{equation}
The function $|\Pi(x,x)|$ is called the Bergman function of $\HoL$. It can be 
expressed as
\begin{equation}\label{EqBergmanBasis}
| \Pi(x,x) | = \sum_{j=1}^N \big | s_j (x)\big |^2,
\end{equation}
and it satisfies
\begin{equation}\label{EqBergmanDiag}
| \Pi(x,x) | = \int_X |\Pi(x,y)|^2 \, dV(y).
\end{equation}

\begin{lemma}\label{LemBergmanSec}
Let $y\in X$. There is a section $\Phi_y \in \HoL$ such that
\begin{equation*}
\big | \Phi_y (x) \big |= \big | \Pi (x,y)\big |, \quad x\in X.
\end{equation*}
\end{lemma}

\begin{proof}
Let $s_1,\ldots, s_N$ be an orthonormal basis for $\HoL$.
Fix a frame $e(x)$ in a neighborhood $U$ of the point $y$, then
in this neighborhood each $s_j$ is represented by a holomorphic
function $f_j$ such that $s_j (x)= f_j (x) e(x)$. Define
\begin{equation*}
\Phi_y (x) := |e(y)| \sum_{j=1}^n \overline{f_j (y)} s_j (x),
\end{equation*}
then $\Phi_y$ is a holomorphic section to $L$, and we have
\[
\big | \Phi_y (x)\big | = \bigg| \Big(
\sum_{j=1}^N  \overline{f_j(y)} s_j(x) \Big) \otimes \overline{e(y)} \bigg |
=\bigg |
\sum_{j=1}^N  s_j (x) \otimes \overline{s_j(y)}\bigg | =\big|\Pi (x,y)\big|.
\qedhere
\]
\end{proof}

We denote by $\Pi_k (x,y)$ the
Bergman kernel for the $k$'th power $L^k$ of the line bundle $\Lcal$ (where $L^k$ is
endowed with the product metric $k \phi$). The behavior of 
$\Pi_k (x,y)$ as $k \to \infty$ 
is of special importance. In the case when the line bundle $(L, \phi)$ is positive,
it is known (see e.g.\ \cite{Berndtsson03, Lindholm01}) that
\begin{equation}\label{EstimateOnDiagonal}
k^n \lesssim |\Pi_k(x,x)| \lesssim k^n,
\end{equation}
and
\begin{equation}\label{EstimateOffDiagonal}
|\Pi_k(x,y)| \lesssim k^n \exp \big( -c \sqrt{k} \, d(x,y) \big),
\end{equation}
where $c = c(X,\omega,L,\phi)$ is an appropriate positive constant.

\subsection{Sub-mean value property}

Let $s\in \HoLk$. If $z \in X$ and $0 < \delta < 1$, then
\begin{equation}\label{InEqSubMean}
\big |s(z)\big |^p \lesssim_p \bigl(\frac \delta{\sqrt{k}}\bigr)^{-2n} 
\int_{B(z, \delta/\sqrt{k})} \big |s(x)\big |^p \qquad\quad (1\leq p< \infty).
\end{equation}
where by $\lesssim_p$ we mean that the implicit constant may also 
depend on $p$.
This can be deduced easily from the compactness of $X$ and the
corresponding fact in $\Cbb^n$, which may be found for example in
\cite[Lemma 7]{Lindholm01}.

As a consequence we have the following Plancherel-P\'olya type inequality:

\begin{lemma}\label{LemmaPlancherelPolya}
Let $\{x_j\}$ be  points in $X$ such that
$d(x_i, x_j) \geq \delta/\sqrt{k}$, where $0 < \delta < 1$. Then
\begin{equation*}
k^{-n} \sum_{j}
\big | s(x_j) \big |^p \lesssim_p
\delta^{-2n} \int_X \big |s(x) \big |^p  \qquad\quad (1\leq p< \infty)
\end{equation*}
for any $s \in \HoLk$.
\end{lemma}


\section{Fekete points and their properties}\label{SecFeketeProperties}

\subsection{}
Let $N = \dim \HoL$, and $s_1, \ldots, s_N$ be a basis for $\HoL$.
A configuration of $N$ points $x_1,\ldots, x_N$ in $X$ is called a \emph{Fekete
configuration} if it maximizes the pointwise norm of the Vandermonde-type determinant
\begin{equation*}
\det (s_i (x_j)), \quad 1\leq i,j\leq N,
\end{equation*}
which is a holomorphic section to the line bundle $\Lcal^{\boxtimes N}$ over the
manifold $X^N$ (endowed with the metric inherited from $\Lcal$).

If $e_j(x)$ is a frame in a neighborhood $U_j$ of the point $x_j$, then
the sections $s_i (x)$ are represented on each $U_j$ by scalar functions
$f_{ij}$ such that
$s_i (x)=f_{ij} (x) e_j(x)$. Similarly, the metric $\phi$ is represented on $U_j$
by a smooth real-valued function $\phi_j$ such that
$| s_i (x)|^2= |f_{ij} (x) |^2 e^{-\phi_j (x)}$.
A Fekete configuration thus maximizes the quantity
\begin{equation}\label{vandermonde}
e^{-\phi_1 (x_1)}\cdots e^{-\phi_N (x_N)}
\big | \det \big ( f_{ij} (x_j)\big )\big |^2.
\end{equation}

By the compactness of $X$, Fekete configurations exist, but in general 
there need not be a unique one. One may check that the norm $| \det (s_i (x_j)) |_\phi$
at a Fekete configuration $x_1, \dots, x_N$ is always non-zero.
It is also easy to check that the definition of a Fekete configuration does not
depend on the particular choice of the basis $s_1, \ldots,s_N$ of $\HoL$.

The function \eqref{vandermonde} is a Vandermonde-type determinant that
vanishes when two points are equal. It is exactly the familiar
Vandermonde determinant in the special case when the sections $s_i$ 
are the monomials in dimension one, and the weight $\phi$ is constant.
This suggests what is actually happening -- the Fekete points
repel each other and tend to be in a sense ``maximally spread''.

\subsection{}
The main property of the Fekete points $x_1,\ldots, x_N$ that will be
used is the existence of ``Lagrange sections" with a uniformly bounded norm.
Namely, we have sections $\ell_1, \dots, \ell_N$ in $\HoL$ such that
\begin{equation}\label{EqLagrangeDelta}
| \ell_j (x_i) | =\delta_{ij}, \quad 1\leq i,j\leq N,
\end{equation}
and moreover, they satisfy the additional condition
\begin{equation}\label{EqLagrangeBound}
\sup_{x\in X} \big | \ell_j (x) \big |=1, \quad 1\leq j\leq N.
\end{equation}
To construct these sections we denote by $M$ the matrix
$\big ( e^{-\frac{1}{2}\phi_j (x_j)}f_{ij}(x_j)\big )$, and define
\begin{equation*}
\ell_j (x) := \frac{1}{\det (M)} \sum_{i=1}^N (-1)^{i+j} M_{ij} s_i (x),
\end{equation*}
where $M_{ij}$ is the determinant of the submatrix obtained from $M$ by removing
the $i$-th row and $j$-th column. Clearly $\ell_j\in \HoL$, and it is not 
difficult
to check that conditions \eqref{EqLagrangeDelta} and \eqref{EqLagrangeBound} above hold,
where \eqref{EqLagrangeBound} is a
consequence of the extremal property of the Fekete configuration $x_1,\ldots,
x_N$.

We also observe that the system $\{\ell_j (x)\}$ forms a \emph{basis} of $\HoL$.
Indeed, the condition \eqref{EqLagrangeDelta} implies that the $\ell_j(x)$ are 
linearly independent, and
since they form a system with $N$ elements, $N=\dim \HoL$,
they span the whole $\HoL$. 
An element $s\in \HoL$ thus has a unique expansion
\begin{equation*}
s(x)=\sum_{j=1}^N c_j \ell_j (x),
\end{equation*}
and the coefficients $c_j$ are given by
\begin{equation*}
c_j=\big \langle s(x_j), \ell_j (x_j)\big \rangle, \quad 1\leq j\leq N,
\end{equation*}
which again follows from \eqref{EqLagrangeDelta}.

\subsection{}
One consequence of the construction above
is that Fekete points form a separated array.

\begin{lemma}\label{LemFekSeparated}
Let $\Fcal_k$ be a Fekete configuration for $(L^k, k \phi)$. Then
\begin{equation}\label{EqFekSep}
d(x,y) \gtrsim \frac1{\sqrt{k}}, \quad x,y \in \Fcal_k, \quad x \neq y.
\end{equation}
\end{lemma}

\begin{proof}
Indeed, if this is not the case, there are points $x_k, y_k\in \Fcal_k$,
$\sqrt{k}\,d(x_k,y_k)\rightarrow 0$ but $x_k \neq y_k$, for infinitely many
$k$'s.
By compactness we may assume that $x_k, y_k$ converge to some point $x\in X$. 
We choose local coordinates $z$ in a neighborhood of $x$, and a local
trivialization of the line bundle $\Lcal$ in this neighborhood. The metric on
$\Lcal$ is represented by a smooth function $\phi (z)$, and the metric on
$\Lcal^k$ is given by $k \phi (z)$.

For each $k$, we have a ``Lagrange section" vanishing on $x_k$ and having norm
one on $y_k$. Let it be given by a holomorphic function $f_k (z)$ with respect
to the local trivialization. Thus
\begin{equation*}
\big | f_k (z)\big |^2 e^{-k\phi (z)} =\begin{cases}
0, & z=z(x_k)\\*[4pt]
1, & z=z(y_k)
\end{cases}
\end{equation*}
and $| f_k (z)|^2 e^{-k\phi (z)}\leq 1$ for all other $z$.

On the other hand, 
the distance function $d$ is equivalent to the Euclidean distance with respect
to the local coordinates. Hence,
\begin{equation*}
\sqrt{k} \,\big | z(x_k)-z(y_k)\big | \longrightarrow 0 \quad (k \rightarrow
\infty).
\end{equation*}
This implies that the norm of the gradient of $|f|^2 e^{-k\phi}$ must be,
at some point $z_k$, larger that $\sqrt{k}$ times a magnitude tending to infinity.
However, Lemma \ref{LemGradient} below shows that the last conclusion is not possible,
and this contradiction concludes the proof of Lemma \ref{LemFekSeparated}.
\end{proof}

\begin{lemma}\label{LemGradient}
Let $\phi (z)$ be a smooth, real-valued function in a neighborhood of the
point $w=(w_1,\ldots,w_n)\in \Cbb^n$. Then there are constants $C$ and $k_0$ 
such that the 
following holds. Let $k \geq k_0$, and $f(z)$ be a holomorphic function
in a neighborhood of the compact set
\[
U_k(w)=\big \{z\in \Cbb^n : |z_j-w_j|\leq 1/\sqrt{k} \; (j=1,\dots,n) \big \}.
\]
Then for each $1\leq j\leq n$ we have
\begin{equation*}
\bigg | \frac{\partial}{\partial z_j} \Big [ \big | f\big |^2 e^{-k\phi}\Big ]
(w) \bigg |\leq C \sqrt{k}\sup_{U_k(w)} \big |f\big |^2 e^{-k\phi}.
\end{equation*}
\end{lemma}

This is proved in dimension one in \cite[Lemma~19(b)]{AmeOrt12}. The 
multi-dimensional
version above can be proved in a similar way.

If the line bundle $(L, \phi)$ is positive, the separation condition
\eqref{EqFekSep} of the Fekete array is sharp in a sense.
The following is true.

\begin{lemma}\label{LemRelDense}
If $(L, \phi)$ is positive then there is $R > 0$ not depending on $k$, 
with the following property: 
if $\Fcal_k$ is a Fekete configuration for $(L^k, k \phi)$, then any ball
$B(x,R/\sqrt{k})$, $x \in X$, contains at least one point of 
$\Fcal_k$.
\end{lemma}

This result may be deduced from Theorem~\ref{necessary_conditions} and
Lemma~\ref{LemFekSamInt} below. However, as it will not be 
used later on, we do not present the details of the proof. We merely state
it to show that the Fekete points $\Fcal_k$ are roughly spread away from
each other at a distance $1/\sqrt{k}$.

\section{Sampling and Interpolation arrays}\label{secsampintFek}

\subsection{}
In this section we relate the Fekete arrays to the sampling and interpolation
arrays. We will show that if the line bundle $(L, \phi)$ is positive,
then by a ``small perturbation'' of the Fekete array one obtains a
sampling or interpolation array for $(L, \phi)$.

\begin{definition}
Let $k$ be a positive integer, and $\Lambda_k$ be a finite set of points in $X$.
We say that $\Lambda_k$ is a \emph{sampling set at level $k$ with sampling
constants $A,B$} if the inequalities
\begin{equation}\label{DefSampling}
A k^{-n} \sum_{\lambda\in \Lambda_k} |s(\lambda)|^2
\leq \int_X \big | s(x)\big|^2  \leq B k^{-n}
\sum_{\lambda\in \Lambda_k} |s(\lambda)|^2
\end{equation}
hold for any section $s\in \HoLk$.
We say that $\Lambda_k$ is an \emph{interpolation set at level $k$ with
interpolation constant $C$} if for any set of values
$\{v_\lambda\}_{\lambda\in \Lambda_k}$, where each $v_\lambda$ is an
element of the fiber of $\lambda$ in $\Lcal^k$, there is a section
$s\in \HoLk$ such that $s(\lambda)=v_\lambda$ $(\lambda\in \Lambda_k)$
and
\begin{equation}\label{DefInterpEst}
\int_X \big | s(x) \big |^2  \leq C k^{-n}
\sum_{\lambda \in\Lambda_k} |v_\lambda|^2.
\end{equation}
\end{definition}

\begin{definition}
Let $\Lambda=\{\Lambda_k\}$ be an array of points, i.e.\ a sequence
of finite sets $\Lambda_k$ in $X$. We call $\Lambda$ a \emph{sampling array}
if there are $k_0$ and positive constants $A,B$ not depending on $k$,
such that $\Lambda_k$ is a sampling set at each level $k \geq k_0$
with sampling constants $A,B$.
Analogously, $\Lambda$ is an \emph{interpolation array}
if there are $k_0$ and a positive constants $C$ not depending on $k$,
such that $\Lambda_k$ is an interpolation
set at each level $k \geq k_0$ with interpolation constant $C$.
\end{definition}

\begin{lemma}\label{LemFekSamInt}
Suppose that $(L, \phi)$ is positive.
Let $k$ be a positive integer, and $\eps$ be a number
satisfying $1/k \lesssim \varepsilon \lesssim 1$. If we define
\[
\Lambda_k := \Fcal_{(1+\varepsilon)k}
\]
then $\Lambda_k$ is a sampling set at level $k$ with sampling constants
$A,B$ such that $1 \lesssim A < B \lesssim \varepsilon^{-2n}$. 
On the other hand, if
\[
\Lambda_k := \Fcal_{(1-\varepsilon)k}
\]
then it is an interpolation set at level $k$ with interpolation constant $C$
satisfying $C \lesssim \varepsilon^{-2n}$.
\end{lemma}

We must provide a clarification concerning the statement of the theorem:
we have written $\Fcal_{(1 \pm \eps)k}$ as if the numbers 
$(1\pm \eps)k$ were integers. In practice, the reader should replace
these numbers by an integer approximation. The same is true in other
parts of the paper below, where we shall keep using such notation.

It follows from Lemma \ref{LemFekSamInt}
that by a ``small perturbation'' of the Fekete array one
obtains a sampling or interpolation array for $(L, \phi)$.

\begin{corollary}\label{cor_pert_fek}
Let $(L, \phi)$ be positive, and $\eps > 0$ be fixed. Then
\begin{enumerate}
\addtolength{\itemsep}{4pt}
\item[{\rm (i)}]
$\{\Fcal_{(1+\varepsilon)k}\}$ is a sampling array for $(L, \phi)$;
\item[{\rm (ii)}]
$\{\Fcal_{(1-\varepsilon)k}\}$ is an interpolation array for $(L, \phi)$.
\end{enumerate}
\end{corollary}

The rest of this section is devoted to the proof of Lemma \ref{LemFekSamInt}.

\subsection{}
We start with the interpolation part of Lemma \ref{LemFekSamInt}.
We fix $k$ and $\eps$ satisfying $1/k \lesssim \varepsilon \lesssim 1$
and define the set $\Lambda_k = \Fcal_{(1-\varepsilon)k}$. We will prove that
$\Lambda_k$ is an interpolation set at level $k$ with interpolation
constant $C$ satisfying $C \lesssim \varepsilon^{-2n}$.

Denote by $\{x_j\}$ the elements of the finite
set $\Lambda_k$. Since the points $\{x_j\}$ form a Fekete configuration for the line
bundle $L^{(1-\varepsilon)k}$, they have associated Lagrange sections
$\ell_j$ (see Section \ref{SecFeketeProperties}). 
The sections $\ell_j$ are suitable
for solving the interpolation problem with nodes $x_j$, but we also need
an estimate for the $L^2$ norm of the solution. For this reason
we need to improve the localization of $\ell_j$ around the point $x_j$.
We therefore define the auxiliary sections 
\[
Q_j (x):= \ell_j (x) \otimes \left [
\frac{\Phi_{x_j}^{(\varepsilon/2)k}(x)}{|\Pi_{(\varepsilon/2)k}(x_j,x_j)|}\right ]^2
\in \HoLk,
\]
where $\Phi_y^{(\varepsilon/2)k}$ denotes a holomorphic section
to $\Lcal^{(\varepsilon/2)k}$ such that
\begin{equation}\label{EstimateQjPhi}
| \Phi_{y}^{(\varepsilon/2)k} (x) | = | \Pi_{(\varepsilon/2)k}
(x,y)|, \quad x\in X.
\end{equation}
The existence of such a section is guaranteed by Lemma \ref{LemBergmanSec}. 

We have thus constructed sections $Q_j$ in $\HoLk$ which are
associated to the points $\{x_j\}$.
Similar to the Lagrange sections, the sections $Q_j$ satisfy
\begin{equation}\label{EqQjDelta}
| Q_j (x_i) | =\delta_{ij},
\end{equation}
as follows from \eqref{EqBergmanDiag} and \eqref{EqLagrangeDelta}.
We will also need the additional estimates
\begin{equation}\label{EstimateQjInt}
\sup_j \int_X |Q_j(x)| \lesssim (\eps k)^{-n},
\end{equation}
and
\begin{equation}\label{EstimateQjSum}
\sup_{x \in X} \sum_j |Q_j(x)| \lesssim \eps^{-n},
\end{equation}
that will be proved now.
The inequality \eqref{EstimateQjInt} follows directly from
\eqref{EqBergmanDiag}, \eqref{EstimateOnDiagonal} and
\eqref{EqLagrangeBound}.
To prove \eqref{EstimateQjSum} we recall that Fekete points are 
separated (Lemma \ref{LemFekSeparated}), and hence
\[
d(x_i, x_j) \gtrsim \frac1{\sqrt{(1-\eps)k}} \gtrsim \frac{\delta}{\sqrt{(\eps/2) k}}
\]
with $\delta = \sqrt{\eps}$. Thus an application of the Plancherel-P\'olya inequality
(Lemma \ref{LemmaPlancherelPolya}) to the section
$\Phi_{x}^{(\eps/2) k}$ and to the set of points $\{x_j\}$ yields
\[
\sum_j |Q_j(x)| \lesssim (\eps k)^{-2n} \sum_j |\Phi_{x}^{(\eps/2) k}(x_j)|^2
\lesssim \eps^{-2n} k^{-n} \int_X |\Phi_{x}^{(\eps/2) k}|^2
\lesssim \eps^{-n},
\]
where in these inequalities we have used 
\eqref{EqBergmanSymm}, \eqref{EqBergmanDiag}, \eqref{EstimateOnDiagonal}
and \eqref{EstimateQjPhi}.

We are now ready to solve the interpolation problem with estimate.
Suppose that we are given a set of values $\{v_j\}$,
where each $v_j$ is an element of the fiber of $x_j$ in $\Lcal^k$.
We will construct a solution $Q(x)$ to the interpolation problem,
i.e.\ a section $Q \in \HoLk$ such that
$Q(x_j)=v_j$ for all $j$. The solution is defined as a linear
combination of the $Q_j$,
\begin{equation*}
Q(x)=\sum_j c_j Q_j (x),
\end{equation*}
with the coefficients $c_j$ given by $c_j= \langle v_j, Q_j(x_j)\rangle$.
This choice of coefficients and the property \eqref{EqQjDelta}
imply that $Q(x)$ is indeed a solution to the interpolation problem.

It remains to show that the solution $Q(x)$ is 
bounded in $\Lcal^2$ with the estimate
\begin{equation}\label{EqQEstimateNeed}
\int_X \big | Q(x)\big |^2 \lesssim \eps^{-2n} k^{-n} \sum_j |v_j|^2.
\end{equation}
Indeed, by the Cauchy-Schwartz inequality and \eqref{EstimateQjSum} we have
\begin{equation}\label{IneqQjCauchySchw}
|Q(x)|^2 \leq \Big( \sum_j |c_j|^2 |Q_j(x)| \Big) \Big( \sum_j |Q_j(x)| \Big)
\lesssim \eps^{-n} \sum_j |c_j|^2 |Q_j(x)|.
\end{equation}
Integrating over $X$ and using \eqref{EstimateQjInt} yields
\[
\int_X |Q(x)|^2 \lesssim \eps^{-n} \sum_j |c_j|^2 \int_X |Q_j(x)|
\lesssim \eps^{-2n} k^{-n} \sum_j |c_j|^2,
\]
and since $|c_j| = |v_j|$ this gives \eqref{EqQEstimateNeed}.

This complete the proof of the interpolation part of Lemma \ref{LemFekSamInt}.

\subsection{}
We continue to the proof of the sampling part of Lemma \ref{LemFekSamInt}.
In this case we are dealing with the set
$\Lambda_k = \Fcal_{(1+\varepsilon)k}$, and must prove that
it is a sampling set at level $k$ with sampling constants
$A,B$ such that $1 \lesssim A < B \lesssim \varepsilon^{-2n}$. 

Again we denote by $\{x_j\}$ the elements of $\Lambda_k$.
We will prove the sampling inequality
\begin{equation}\label{EqSampEstimateNeed}
k^{-n} \sum_{j} |s(x_j)|^2 \lesssim
\int_X \big | s(x)\big|^2 \lesssim \eps^{-2n} k^{-n} 
\sum_{j} |s(x_j)|^2
\end{equation}
for any section $s \in \HoLk$.
The left hand side of \eqref{EqSampEstimateNeed} is a consequence
of the Plancherel-P\'olya inequality 
(Lemma \ref{LemmaPlancherelPolya}) and the separation condition
\[
d(x_i, x_j) \gtrsim \frac1{\sqrt{(1+\eps)k}}
\gtrsim \frac1{\sqrt{k}}
\]
ensured by Lemma \ref{LemFekSeparated}.

The proof of the right hand side of \eqref{EqSampEstimateNeed} is
similar to the interpolation part. Fix $x\in X$ and define
\[
P_x (y):= s(y) \otimes \left [
\frac{\Phi_{x}^{(\varepsilon/2)k}(y)}{|\Pi_{(\varepsilon/2)k}(x,x)|}\right ]^2
\in H^0(\Lcal^{(1+\varepsilon)k}).
\]
The space $H^0(\Lcal^{(1+\varepsilon)k})$ has a basis
of Lagrange sections $\ell_j$ associated to the Fekete
points $\{x_j\}$, so we may expand $P_x$ in terms of this basis. We get
\begin{equation*}
P_x(y)=\sum_j \big \langle P_x(x_j), \ell_j(x_j) \big \rangle\, \ell_j(y).
\end{equation*}
In particular, if $y = x$ this implies
\begin{equation*}
| s(x) | = |P_x (x) | \leq \sum_j |P_x(x_j )|
= \sum_j |s(x_j )|\, |Q_j(x)|,
\end{equation*}
where now we define
\[
Q_j(x) := \left [
\frac{\Phi_{x_j}^{(\varepsilon/2)k}(x)}{|\Pi_{(\varepsilon/2)k}(x,x)|}\right ]^2.
\]
The estimates \eqref{EstimateQjInt}, \eqref{EstimateQjSum} are valid
in this case as well, and can be proved in the same way. We may therefore
continue as in \eqref{IneqQjCauchySchw}. We obtain
\begin{equation}
|s(x)|^2 \leq \Big( \sum_j |s(x_j)|^2 |Q_j(x)| \Big) \Big( \sum_j |Q_j(x)| \Big)
\lesssim \eps^{-n} \sum_j |s(x_j)|^2 |Q_j(x)|,
\end{equation}
and integrating over $X$ yields the right hand side of \eqref{EqSampEstimateNeed}.

We have thus proved also the sampling part of Lemma \ref{LemFekSamInt}, so
the lemma is completely proved.

\begin{remark}
In the proof of Lemma \ref{LemFekSamInt} we have not used any
off-diagonal estimate such as \eqref{EstimateOffDiagonal} for
the Bergman kernel, but only the asymptotic estimate
\eqref{EstimateOnDiagonal} on the diagonal combined with
the $L^2$ equality \eqref{EqBergmanDiag}
(this is in contrast to \cite{AmeOrt12}, for example).
\end{remark}


\section{Landau's inequalities}\label{Landconc}

\subsection{}
In this section we use Landau's method \cite{Landau67}
to obtain estimates for the 
number of points of a separated sampling or interpolation array
in a ball.

Let us say that a finite set of points $\Lambda_k$ is \emph{$\delta$-separated
at level $k$} if 
\begin{equation}\label{SepDelta}
d(x,y) \geq \frac{\delta}{\sqrt{k}}, \quad x,y \in \Lambda_k, \quad x \neq y.
\end{equation}

Our goal is to prove the following two statements.

\begin{lemma}\label{lem_landau_sampling}
Let $\Lambda_k$ be a $\delta$-separated sampling set at level $k$
with sampling constants $A,B$. Then for any $z \in X$ and $r >0$,
\begin{equation}
\label{landau_sampling}
\# \big( \Lambda_k \cap B \big( z, \tfrac{r+\delta}{\sqrt{k}} \,\big ) \big)
\geq \int_\Omega |\Pi_k(x,x)| - M \iint_{\Omega\times\Omega^c}
\big |\Pi_k(x,y) \big |^2,
\end{equation}
where $\Omega=B(z, \frac{r}{\sqrt{k}})$, and the constant
$M$ is bounded by the sampling constant $B$
times a constant which may depend on $\delta$ but does not depend
on $k,z,r$.
\end{lemma}

\begin{lemma}\label{lem_landau_interpolation}
Similarly, if $\Lambda_k$ is a $\delta$-separated interpolation set at 
level $k$ with interpolation constant $C$, then for any $z \in X$ and $r >0$,
\begin{equation}
\label{landau_interpolation}
\# \big( \Lambda_k \cap B \big( z, \tfrac{r-\delta}{\sqrt{k}} \,\big ) \big)
\leq \int_\Omega |\Pi_k(x,x)| + M \iint_{\Omega\times\Omega^c}
\big |\Pi_k(x,y) \big |^2,
\end{equation}
where again $\Omega=B(z, \frac{r}{\sqrt{k}})$,
and the constant $M$ is bounded by the interpolation constant $C$
times a constant which may depend on $\delta$ but does not depend
on $k,z,r$.
\end{lemma}

\subsection{}
Let $\Omega$ be a measurable subset of $X$. We denote by $T_\Omega$
the linear operator on $H^0(L)$ defined by
\begin{equation*}
T_\Omega (s)=P(s \cdot \1_\Omega), \quad s\in H^0(L),
\end{equation*}
where $P$ denotes the orthogonal projection from the Hilbert space of all
$L^2$ sections onto its finite-dimensional subspace $H^0(L)$. It is easy
to see that
\begin{equation*}
\big \langle T_\Omega s, s \big \rangle=\int_\Omega |s|^2,  \quad s\in H^0(L),
\end{equation*}
hence $T_\Omega$ is self-adjoint, non-negative and $\|T_\Omega\|\leq 1$.
We may therefore find an orthonormal basis $\{s_j\}$ of $H^0(L)$
consisting of eigensections,
\begin{equation*}
T_\Omega (s_j)=\lambda_j (\Omega)s_j.
\end{equation*}
The eigenvalues $\lambda_j (\Omega)$ lie between $0$ and $1$, and we
order them in a non-increasing order,
\begin{equation*}
\lambda_1(\Omega) \geq \lambda_2(\Omega) \geq \lambda_3(\Omega) \geq
\cdots \geq 0.
\end{equation*}
By using \eqref{EqBergmanBasis} with the basis of eigensections $\{s_j\}$
we can compute the trace of $T_\Omega$,
\begin{equation}\label{EqTrace}
\sum_{j\geq 1} \lambda_j (\Omega) =
\sum_{j\geq 1} \big\langle T_\Omega s_j, s_j\big\rangle =
\sum_{j\geq 1} \int_\Omega \big | s_j(x)\big |^2 =
\int_\Omega |\Pi(x,x)|.
\end{equation}
Similarly, 
\eqref{EqBergmanKernelBasis} allows us to compute the Hilbert-Schmidt norm
of $T_\Omega$ (the trace of $T_\Omega^2$) in terms of the Bergman kernel.
Indeed, 
\begin{equation*}
\big | \Pi (x,y)\big |^2=\sum_{j\geq 1} \sum_{k\geq 1}\big \langle s_j (x), s_k
(x) \big \rangle 
\overline{\big\langle s_j (y), s_k (y)\big \rangle},
\end{equation*}
hence integrating over $\Omega\times \Omega$ gives
\begin{equation}\label{EqHSNorm}
\sum_{j\geq 1} \lambda_j^2 (\Omega) =
\sum_{j,k}\big | \langle T_\Omega s_j, s_k\rangle \big |^2=
\sum_{j,k} \bigg | \int_\Omega \langle s_j,s_k\rangle \bigg |^2=
\iint_{\Omega\times\Omega} \big | \Pi (x,y) \big |^2.
\end{equation}

Using \eqref{EqTrace} and \eqref{EqHSNorm} one may obtain some 
information on the distribution of the eigenvalues. This is done
in the following lemma.

\begin{lemma}
Let $0<\gamma <1$ and denote by $n (\Omega,\gamma)$ the number of eigenvalues
$\lambda_j(\Omega)$ which are strictly greater than $\gamma$. Then we have
the lower bound
\begin{equation}\label{EigenBelow}
n(\Omega,\gamma)\geq \int_\Omega |\Pi(x,x)|
-\frac{1}{1-\gamma}\underset{\Omega\times\Omega^c}{\iint}|\Pi (x,y)|^2,
\end{equation}
and the upper bound
\begin{equation}\label{EigenAbove}
n(\Omega,\gamma)\leq \int_\Omega |\Pi(x,x)|
+\frac{1}{\gamma}\underset{\Omega\times\Omega^c}{\iint}|\Pi (x,y)|^2.
\end{equation}
\end{lemma}

\begin{proof}
We have
\begin{equation*}
\1_{(\gamma, 1]} (x) \geq x-\frac{x(1-x)}{1-\gamma} \quad (0\leq x\leq 1),
\end{equation*}
hence
\begin{align*}
 n (\Omega,\gamma) = \sum_j \1_{(\gamma,1]} \big ( \lambda_j(\Omega)\big ) \geq
\sum_j \lambda_j (\Omega)-\frac{1}{1-\gamma} \sum_j \big ( \lambda_j
(\Omega)-\lambda_j^2 (\Omega)\big ).
\end{align*}
Using \eqref{EqTrace},\eqref{EqHSNorm} and \eqref{EqBergmanDiag} this implies
\begin{align*}
 n (\Omega,\gamma) &\geq
 \int_\Omega |\Pi(x,x)| -\frac{1}{1-\gamma} \left [  \int_\Omega |\Pi(x,x)| -
\underset{\Omega\times\Omega}{\iint} \big | \Pi (x,y)\big |^2\right ] =\\*[4pt]
 &= \int_\Omega |\Pi(x,x)| -\frac{1}{1-\gamma} \left [  \iint_{\Omega\times X} \big
|\Pi (x,y)\big |^2 - \underset{\Omega\times\Omega}{\iint} \big | \Pi (x,y)\big
|^2\right ]
\end{align*}
which proves (i). To prove (ii) one may argue similarly using the inequality
\begin{equation*}
\1_{(\gamma,1]} (x)\leq  x +\frac{x(1-x)}{\gamma} \quad (0\leq x\leq 1).
\qedhere
\end{equation*}
\end{proof}

\subsection{}
Now consider powers $L^k$ of the line bundle $L$. We obtain an
operator $T_\Omega^{(k)}$ acting on $\HoLk$ with corresponding eigenvalues
\begin{equation*}
\lambda_1^{(k)} (\Omega) \geq \lambda_2^{(k)} (\Omega) \geq \cdots \geq 0,
\end{equation*}
and we let $n_k (\Omega,\gamma)$ denote the number of eigenvalues
strictly greater than $\gamma$ $(0<\gamma <1)$.

\begin{lemma}
Let $\Lambda_k$ be a $\delta$-separated sampling set at level $k$
with sampling constants $A,B$. Then for any $z\in X$ and $r>0$, 
\begin{equation*}
\# \big( \Lambda_k \cap B \big( z, \tfrac{r+\delta}{\sqrt{k}} \,\big ) \big)
\geq n_k \big( B \big( z, \tfrac{r}{\sqrt{k}} \,\big ), \gamma \big)
\end{equation*}
where $\gamma$ is some constant lying between $0$ and $1$, such that
$1/(1-\gamma)$ is bounded by the sampling constant $B$
times a constant which may depend on $\delta$ but does not depend
on $k,z,r$.
\end{lemma}

\begin{proof}
 Let $\{s_j\}$ be the orthonormal basis of $\HoLk$ which is associated to the
eigenvalues $\lambda_j^{(k)}(\Omega)$, where $\Omega=B \big( z, \tfrac{r}{\sqrt{k}} \,\big )$.
Let $N:=
\# \big(\Lambda_k \cap B(z,\frac{r+\delta/2}{\sqrt{k}})\big)$.
 We may restrict to the case when $N$ is strictly smaller than
 $\dim \HoLk$, since otherwise the inequality holds trivially.
 In this case, we may choose a linear combination
 \begin{equation*}
s=\sum_{j=1}^{N+1} c_j s_j
\end{equation*}
of the first $N+1$ eigensections, such that
\begin{equation*}
s(\lambda)=0, \quad \lambda \in \Lambda_k \cap B \big ( x,
\tfrac{r+\delta/2}{\sqrt{k}}\big )
\end{equation*}
and the $c_j$ are not all zero.
Since $\Lambda_k$ is a sampling set, we have
\begin{equation*}
\|s\|^2 \leq B k^{-n} \sum_{\lambda\in \Lambda_k}
|s(\lambda)|^2 = Bk^{-n} \sum_{\lambda \in \Lambda_k
\setminus B \big (x, \frac{r+\delta/2}{\sqrt{k}} \big)}
|s(\lambda)|^2.
\end{equation*}
Using the inequality \eqref{InEqSubMean} 
and the fact that $B(\lambda, \frac{\delta/2}{\sqrt{k}})$ are disjoint balls, we get
\begin{align*}
\|s\|^2 &\leq K B \sum_\lambda \int_{B\big(\lambda,
\frac{\delta/2}{\sqrt{k}}\big)} |s|^2 \leq K B
\int_{X\setminus \Omega} |s|^2,
\end{align*}
where the constant $K$ may depend on $\delta$ but does not depend
on $k,z,r$. This implies
\begin{align*}
 \lambda_{N+1} (\Omega)
\, \|s\|^2 &= \lambda_{N+1} \sum_1^{N+1} |c_j|^2 \leq \sum_1^{N+1} \lambda_j
|c_j|^2=\big \langle T_\Omega^{(k)} s, s\big \rangle =\int_\Omega |s|^2 \leq 
\gamma \|s\|^2,
\end{align*}
where $\gamma := 1-(K B)^{-1}$. This shows that
$\lambda_{N+1} (\Omega)\leq \gamma$ and hence $n_k(\Omega,\gamma) \leq N$.
\end{proof}

\begin{lemma}
Let $\Lambda_k$ be a $\delta$-separated interpolation set at level $k$
with interpolation constant $C$. Then for any $z\in X$ and $r>0$, 
\begin{equation*}
\# \big( \Lambda_k \cap B \big( z, \tfrac{r-\delta}{\sqrt{k}} \,\big ) \big)
\leq n_k \big( B \big( z, \tfrac{r}{\sqrt{k}} \,\big ), \gamma \big)
\end{equation*}
where $\gamma$ is some constant lying between $0$ and $1$, such that
$1/\gamma$ is bounded by the interpolation constant $C$
times a constant which may depend on $\delta$ but does not depend
on $k,z,r$.
\end{lemma}

\begin{proof}
 Let $W$ denote the orthogonal complement in $\HoLk$ of the subspace of sections
 vanishing on $\Lambda_k$.  Since $\Lambda_k$ is an interpolation set at level $k$,
 for any set of values $\{v_\lambda\}_{\lambda \in \Lambda_k}$, where each $v_\lambda$
 is an element of the fiber of $\lambda$ in $L^k$, there is a section $s\in
\HoLk$  such that $s(\lambda)=v_\lambda$ $(\lambda \in \Lambda_k)$ and
 \begin{equation}\label{equation1}
\| s\|^2 \leq C k^{-n} \sum_{\lambda \in \Lambda_k} |s(\lambda)|^2.
\end{equation}
 By taking the orthogonal projection of $s$ 
 onto $W$ we obtain another  solution to the interpolation problem,
 which in addition  belongs to $W$ (the projection neither changes
the  values of $s$ on $\Lambda_k$ nor increases its norm).

 On the other hand, a section in $W$ is uniquely determined
 by its values on $\Lambda_k$, as follows from the definition of $W$.
 Hence if $s$  is an arbitrary section in $W$, then it is the unique interpolant
 in $W$ to the values $\{s(\lambda)\}_{\lambda\in \Lambda_k}$. This implies that
 \eqref{equation1} holds for any $s\in W$.

 Now let us denote by $x_1, \ldots, x_N$ the elements of
 $ \Lambda_k \cap B \big( z, \tfrac{r-\delta}{\sqrt{k}} \,\big ) $.
For each $1\leq j\leq N$ we can find $s_j \in W$ such that
$|s_j(x_j)|=1$ and $s_j$ vanishes on $\Lambda_k \setminus \{x_j\}$.
Certainly, the $s_j$ form a linearly independent set of vectors.
We denote by $F$ the $N$-dimensional linear subspace spanned 
by the sections $s_1, \ldots, s_N$.

Now take any $s\in F$, then we have
\begin{align*}
 \|s\|^2 &\leq C k^{-n} \sum_{\lambda \in \Lambda_k} |s(\lambda)|^2
 = C k^{-n} \sum_{\lambda \in \Lambda_k \cap B \big(x,
\frac{r-\delta/2}{\sqrt{k}}\big)} |s(\lambda)|^2 <
K C \int_\Omega |s|^2,
\end{align*}
where $\Omega=B \big( z, \tfrac{r}{\sqrt{k}} \,\big )$, and
the constant $K$ may depend on $\delta$ but does not depend
on $k,z,r$. The last inequality holds by \eqref{InEqSubMean} 
and the fact that $B(\lambda, \frac{\delta/2}{\sqrt{k}})$ are disjoint balls.
Hence
\begin{equation*}
\frac{\langle T_\Omega^{(k)}s, s\rangle}{\|s\|^2}=\frac{\int_\Omega
|s|^2}{\|s\|^2} > \frac{1}{K C} =: \gamma,
\end{equation*}
for any section $s$ in the $N$-dimensional linear subspace $F$.
By the min-max theorem this implies that
$\lambda_N (\Omega) >\gamma$ and hence $n_k(\Omega, \gamma) \geq N$.
\end{proof}

\section{Curvature and density}\label{curvdens}

In the previous section we have used Landau's method to estimate
the number of points of a sampling or interpolation set
in a ball, where the estimate obtained was given in terms of the Bergman
kernel $\Pi_k(x,y)$. In the present section we will prove
Theorem \ref{necessary_conditions} by relating the latter
estimate to geometric properties of the positive line bundle $(L, \phi)$,
namely, to the volume form associated with the curvature of the 
line bundle.

\subsection{}
Given a point $x\in X$, let $\xi_1, \ldots, \xi_n$ be a basis
for the holomorphic cotangent space at $x$, orthonormal with 
respect to the hermitian metric $\omega$ on $X$. With respect to this basis,
the form $ \partial \bar\partial \phi$ is given at the point $x$ by
\begin{equation*}
\partial \bar\partial \phi =\sum_{j,k} \phi_{j,k} \, \xi_j \wedge \bar{\xi_{k}},
\end{equation*}
where $(\phi_{j,k})$ is a hermitian $n\times n$ matrix.
The eigenvalues $\lambda_1 (x),\ldots, \lambda_n (x)$ of this matrix are 
called \emph{the eigenvalues of the curvature form 
$\partial \bar\partial \phi$ with respect to the hermitian metric $\omega$}.

Recall that the line bundle $L$ with the metric $\phi$ is said to be
positive if $i \partial \bar\partial \phi$ is a positive form. This
is equivalent to all of the eigenvalues $\lambda_1 (x),\ldots, \lambda_n (x)$
being strictly positive, for every $x\in X$. 

If the form $i \partial \bar\partial \phi$ is positive, then the
$(n,n)$-form $(i \partial \bar\partial \phi)^n$ is a volume form
on $X$. Our goal is to provide geometrical information on a sampling or
interpolation array $\Lambda=\{\Lambda_k\}$, by relating the mass
distribution of the measure
\begin{equation*}
k^{-n} \sum_{\lambda \in \Lambda_k} \delta_\lambda
\end{equation*}
to the volume distribution of $(i \partial \bar\partial \phi)^n$
in a quantitative manner. We emphasize that the volume form 
$(i \partial \bar\partial \phi)^n$ is a characteristic of the hermitian 
metric $\phi$ on the line bundle only, 
and does not depend on the arbitrary hermitian metric $\omega$ that we have 
chosen on the manifold $X$.
However, the curvature volume form $(i \partial \bar\partial \phi)^n$ is related
to the volume form $V$ associated with $\omega$ through the eigenvalues, and we have
\begin{equation}
\label{curvature_eigenvalues}
(i \partial \bar\partial \phi)^n = n! \,
\lambda_1(x) \cdots \lambda_n(x) \, dV(x).
\end{equation}

The eigenvalues of the curvature form are related also to the asymptotics
of the Bergman function $|\Pi_k(x,x)|$. When the line bundle
is positive, it was proven in \cite{Tian90}, see \cite{Zelditch98} that
\begin{equation}
\label{bergman_kernel_eigenvalues}
|\Pi_k(x,x)| = \pi^{-n}  \lambda_1(x) \cdots \lambda_n(x) k^n + O(k^{n-1}).
\end{equation}
This a more precise result than \eqref{EstimateOnDiagonal}.
In fact, this is only the first term in a complete asymptotic expansion
obtained in \cite{Zelditch98} into a power series in $k$ (see also
\cite{BerBerSjo08} for a different proof).

\subsection{}
The main ingredient which we need for the proof of Theorem 
\ref{necessary_conditions} is to show that the ``error terms'' in Landau's inequalities
\eqref{landau_sampling} and \eqref{landau_interpolation} are indeed
small with respect to the main term. This is done in the following lemma.

\begin{lemma}
\label{lemma_square_kernel}
Let the line bundle $(L, \phi)$ be positive. If $\Omega =B(z, \frac{r}{\sqrt{k}})$, $\, z\in X$, then
\begin{equation*}
\iint_{\Omega \times \Omega^c} \big | \Pi_k (x,y)\big |^2 \lesssim r^{2n-1}.
\end{equation*}
\end{lemma}

For the proof we will use the asymptotic off-diagonal estimate
\eqref{EstimateOffDiagonal} for the Bergman kernel, which holds when
the line bundle $(L, \phi)$ is positive. In fact, we do not need the precise
exponential decay given by \eqref{EstimateOffDiagonal}.
It will be enough to use the fact that
\begin{equation}\label{offdiag_berg}
|\Pi_k(x,y)|\leq k^n \varphi (\sqrt{k} \,d(x,y)),
\end{equation}
where $\varphi$ is a smooth decreasing function on $[0,\infty)$ such that
\begin{equation}\label{offdiag_phi}
\text{$\varphi (u)=O(u^{-\alpha})$ as $u \to \infty$,
\quad for some $\alpha > n+\tfrac{1}{2}$.}
\end{equation}

\begin{proof}[Proof of Lemma \ref{lemma_square_kernel}]
We partition $\Omega$ into ``dyadic shells" defined by
\begin{equation*}
\Omega_j := \bigg\{
x\in X : \big ( 1-2^{-j+1} \big ) \frac{r}{\sqrt{k}} \leq d(x,z) < \big (
1-2^{-j}\big ) \frac{r}{\sqrt{k}} \bigg \} \quad (j\geq 1).
\end{equation*}
If $x\in \Omega_j$ and $y\in \Omega^c$ then $d(x,y) > 2^{-j}
\frac{r}{\sqrt{k}}$, and thus we have
\begin{equation*}
\iint_{\Omega \times \Omega^c} \big
|\Pi_k (x,y)\big |^2 \leq \sum_{j=1}^\infty \iint_{\Omega_j \times B(x,2^{-j} \frac{r}{\sqrt{k}})^c}  \big |\Pi_k (x,y)\big |^2.
\end{equation*}
To estimate the right hand side we use \eqref{offdiag_berg}. For any $A>0$ we have
\begin{align*}
&\int_{B\big (x, \frac{A}{\sqrt{k}}\big )^c} \big | \Pi_k (x,y)\big
|^2 dV (y) =\\*[4pt] & =\int_0^\infty V \bigg ( \big \{ y : \big
| \Pi_k (x,y) \big | >\lambda \big \} \setminus B \bigg (x,
\frac{A}{\sqrt{k}} \bigg ) \bigg)2\lambda\, d\lambda\\*[4pt]
&\leq \int_0^{k^n \varphi (0)} V \bigg ( \big \{  y : \varphi
\big ( \sqrt{k} \, d(x,y)\big ) \geq k^{-n}\,\lambda \big \}
\setminus B \bigg ( x, \frac{A}{\sqrt{k}}\bigg )
\bigg ) 2\lambda\, d\lambda.
\end{align*}
Since $\varphi$ is decreasing we may apply the change of variable 
$\lambda = k^n \varphi(u)$, and we get
\begin{align*}
&= \int_0^\infty V \bigg (  B \bigg ( x, \frac{u}{\sqrt{k}}\bigg ) \setminus B
\bigg (  x, \frac{A}{\sqrt{k}}\bigg ) \bigg)\, 
 ( 2k^n \varphi (u) ) \, |k^n \varphi' (u)| \, du \\*[4pt]
&\lesssim \int_A^\infty \bigg ( \frac{u}{\sqrt{k}}\bigg )^{2n} \,
 ( 2k^n \varphi (u) ) \, |k^n \varphi' (u)| \, du \\*[4pt]
& \lesssim  k^n \int_A^\infty u^{2n} \, \varphi (u)\, |\varphi' (u)|\, du.
\end{align*}
We also use an estimate for the volume of shells, namely
\begin{equation}
\label{shell_volume}
V \big ( B(x, \rho+\delta) \setminus B(x,\rho)\big ) \lesssim \rho^{2n-1} \,
\delta \quad (0<\delta <\rho),
\end{equation}
which can be proved using the exponential map. In particular, this implies 
\begin{equation*}
V (\Omega_j)\lesssim  2^{-j} \, \frac{r^{2n}}{k^n} \,.
\end{equation*}
Combining all the estimates above yields
\begin{align*}
\iint_{\Omega \times \Omega^c} & \big | \Pi_k (x,y)\big |^2 
\lesssim \sum_{j=1}^\infty \bigg ( 2^{-j}\, \frac{r^{2n}}{k^n}\bigg )
\, k^n \int_{2^{-j} r}^\infty u^{2n} \, \varphi (u) \, |\varphi'(u)|\, du\\*[4pt]
&= r^{2n} \, \int_0^\infty \bigg [ \sum_{j=1}^\infty 2^{-j} \1_{[2^{-j} r,
\infty)} (u)\bigg ]\, u^{2n} \, \varphi (u)\, |\varphi' (u)|\, du\\*[4pt]
&\leq r^{2n} \int_0^\infty \big ( 2 u/r) \,u^{2n}\, \varphi (u)
\, |\varphi' (u)|\, du\\*[4pt]
&\lesssim  r^{2n-1}\, \int_0^\infty u^{2n} \, \varphi (u)^2 \, du,
\end{align*}
where the integration by parts used is justified by \eqref{offdiag_phi}.
Since the last integral converges, again due to \eqref{offdiag_phi}, this
proves the lemma.
\end{proof}

\subsection{}
We can now finish the proof of Theorem \ref{necessary_conditions}.
It is an immediate consequence of the following result.

\begin{lemma}\label{lem_dens_necc}
Let $(L, \phi)$ be positive. If
$\Lambda_k$ be a $\delta$-separated sampling set at level $k$
with sampling constants $A,B$, then for any $z \in X$ and $r >0$,
\begin{equation}
\label{estpt_sampling}
\frac{k^{-n} \# ( \Lambda_k \cap \Omega )}{\int_\Omega 
(i \partial \bar{\partial} \phi)^n} > \frac{1}{\pi^n n!} - \frac{M}{r},
\end{equation}
where $\Omega=B(z, \frac{r}{\sqrt{k}})$, and the constant
$M$ is bounded by the sampling constant $B$
times a constant which may depend on $\delta$ but does not depend on $k,z,r$.

Similarly, if $\Lambda_k$ is a $\delta$-separated interpolation set at 
level $k$ with interpolation constant $C$, then for any $z \in X$ and $r >0$,
\begin{equation}
\label{estpt_interpolation}
\frac{k^{-n} \# ( \Lambda_k \cap \Omega) \big)}{\int_\Omega 
(i \partial \bar{\partial} \phi)^n} < \frac{1}{\pi^n n!} + \frac{M}{r},
\end{equation}
where again $\Omega=B(z, \frac{r}{\sqrt{k}})$, and the constant
$M$ is bounded by the interpolation constant $C$
times a constant which may depend on $\delta$ but does not depend on $k,z,r$.
\end{lemma}

\begin{proof}
Assume first that $\Lambda_k$ is a $\delta$-separated sampling set at level $k$.
Let $\Omega = B \big ( z, \frac{r}{\sqrt{k}}\big )$.
The separation condition together with \eqref{shell_volume} imply
that the number of points of $\Lambda_k$ in the shell 
$B \big ( z, \frac{r+\delta}{\sqrt{k}}\big ) \setminus B 
\big ( z, \frac{r}{\sqrt{k}}\big )$ is less than $M_1 r^{2n-1}$. 
Hence by \eqref{landau_sampling} and Lemma \ref{lemma_square_kernel} we obtain
\begin{equation*}
\# ( \Lambda_k \cap \Omega ) \geq \int_\Omega |\Pi_k(x,x)| - M_2 r^{2n-1}.
\end{equation*}
Using \eqref{curvature_eigenvalues} and \eqref{bergman_kernel_eigenvalues} this implies
\begin{equation*}
\# ( \Lambda_k \cap \Omega ) \geq \frac{k^n}{\pi^n n!}\int_\Omega (i \partial
\bar{\partial} \phi)^n
- M_2 r^{2n-1} - M_3 k^{n-1} V(\Omega).
\end{equation*}
Since $V(\Omega) \lesssim r^{2n}/k^n$ and $r/ \sqrt{k} \leq 
\operatorname{diam} (X)$ it follows that
\begin{equation*}
\# ( \Lambda_k \cap \Omega ) \geq \frac{k^n}{\pi^n n!}\int_\Omega 
(i \partial \bar{\partial} \phi)^n - M_4 r^{2n-1},
\end{equation*}
and since $k^n\int_\Omega (i \partial \bar{\partial} \phi)^n$ is of order $r^{2n}$
this proves the claimed inequality.
In the second case, when $\Lambda_k$ is a $\delta$-separated interpolation set 
at level $k$, the result is proved in a similar 
way using \eqref{landau_interpolation} instead of \eqref{landau_sampling}.
\end{proof}

This concludes the proof of Theorem \ref{necessary_conditions}.

\begin{remark}
One may also define sampling and interpolation arrays with respect to the $L^p$ norm
on the line bundle $(1 \leq p \leq \infty)$. The necessary density conditions 
given in Corollary \ref{densities} could be extended to this setting as well.
This is rather standard and we do not discuss the details, see e.g.\ \cite{Marzo07}.
\end{remark}

\section{Equidistribution of Fekete points}
\label{sec_wasserstein}

In this section we estimate from above and below the number of Fekete points
that lie in a ball. Our proof of this result is inspired by the work of Nitzan
and Olevskii \cite{NitOle12} where they provide a new proof of Landau's
necessary density condition for sampling and interpolation in the Paley-Wiener
space. Their main idea, that we adapt to the study of Fekete points, is to
find a discrete representation of the Bergman kernel on the diagonal as a linear
combination of reproducing kernels on the Fekete points. This produces a
``tessellation'' by functions concentrated around the Fekete points. The same 
technique can be used to provide an upper bound for the Kantorovich-Wasserstein 
distance
between the Fekete measure \eqref{Feketemeasure} and its limiting measure. We
also use the Fekete points to construct a sampling or interpolation array with
density arbitrarily close to the critical one, showing that the necessary
density conditions in Corollary \ref{densities} are sharp.

\subsection{}
To prove Theorems \ref{fekete_distribution} and \ref{wass_estimate} we will need two lemmas. 
The first one is an $L^1$-variant of the off-diagonal decay estimate of the 
Bergman kernel.

\begin{lemma}\label{decay} 
Let the line bundle $(L, \phi)$ be positive. Then
 \begin{enumerate}
\renewcommand\theenumi{\rm (\roman{enumi})}
\renewcommand{\labelenumi}{\theenumi}
\addtolength{\itemsep}{4pt}
  \item\label{bk:i} $\displaystyle \sup_{x\in X} \int_X |\Pi_k(x,y)|\, dV(y)\lesssim 1$;
  \item\label{bk:ii} If $\Omega=B(z,R/\sqrt{k})$ then 
  \[
   k^n \iint_{\Omega\times\Omega^c} |\Pi_k(x,y)| \lesssim R^{2n-1}
  \]
  uniformly in $z\in X$;
  \item\label{bk:iii} 
  $\displaystyle  \sup_{x\in X}\int_X d(x,y) |\Pi_k(x,y)|\, dV(y)\lesssim 1/\sqrt{k}$.
 \end{enumerate}
\end{lemma}
This can be proved with an argument completely similar to the one used
in the proof of Lemma \ref{lemma_square_kernel}, so we omit the details.

\begin{lemma}\label{lem:dual}
Let $\{\ell_\lambda\}$ be the Lagrange sections associated to the Fekete points
$\Fcal_k$. Then there exist sections $\Phi_\lambda\in H^0(L^k)$,
$\lambda\in \Fcal_k$, such that:
\begin{enumerate}
\renewcommand\theenumi{\rm (\roman{enumi})}
\renewcommand{\labelenumi}{\theenumi}
\addtolength{\itemsep}{4pt}
 \item\label{i} $\int_X \langle \ell_\lambda(x),\Phi_\lambda(x)\rangle dV(x)=1,\quad \forall \lambda\in \Fcal_k$.
\item\label{ii} $\sum_{\lambda\in \Fcal_k} \langle
\ell_\lambda(x),\Phi_\lambda(x)\rangle =|\Pi_k(x,x)|,\quad \forall x\in X$.
\item\label{iii} $|\Phi_\lambda(x)|= |\Pi_k(x,\lambda)|$, $\forall x\in X$,
$\lambda\in \Fcal_k$.
\end{enumerate}
\end{lemma}
\begin{proof}
 Let $s_1,\ldots,s_N$ be an orthonormal basis for $H^0(L^k)$. Let
$e_\lambda(x)$ be a holomorphic frame in a neighbourhood $U_\lambda$ of
$\lambda$
($\lambda\in \Fcal_k$). Then 
\[
s_j(x)=f_{j,\lambda}(x) e_\lambda(x),\qquad x\in U_\lambda.
\]
By Lemma~\ref{LemBergmanSec}, if we define
\[
 \Phi_\lambda(x):= |e_\lambda(\lambda)| \sum_{j=1}^N
\overline{f_{j,\lambda}(\lambda)} s_j(x)
\]
then \ref{iii} is satisfied.

We now choose $e_\lambda(x):=\ell_\lambda(x)$, the Lagrange section. Then,
since $|\ell_\lambda(\lambda)|=1$, we have
\[
 \Phi_\lambda(x)= \sum_{j=1}^N \overline{f_{j,\lambda}(\lambda)} s_j(x).
\]
Since $\{s_j(x)\}$ is an orthonormal basis, 
\[
 \ell_\lambda = \sum_{j=1}^N s_j\int_X \langle
\ell_\lambda(x),s_j(x)\rangle\, dV(x).
\]
Therefore 
\[
 \ell_\lambda(\lambda) = \sum_{j=1}^N s_j(\lambda)\int_X \langle
\ell_\lambda(x),s_j(x)\rangle\, dV(x) =
\sum_{j=1}^N f_{j,\lambda} \ell_\lambda(\lambda)\int_X \langle
\ell_\lambda(x),s_j(x)\rangle\,dV(x).
\]
Thus,
\[
 1= \int_X \Bigl\langle
\ell_\lambda(x),\sum_{j=1}^N
\overline{f_{j,\lambda(\lambda)}}s_j(x)\Bigr\rangle\, dV(x)=\int_X
\langle\ell_\lambda(x),\Phi_\lambda(x)\rangle\, dV(x),
\]
which gives \ref{i}.

Since $\{\ell_\lambda\}$ is a ``Lagrange basis'' for $H^0(L^k)$:
\[
 s_j=\sum_{\lambda\in\Fcal_k}\langle s_j(\lambda),\ell_\lambda(\lambda)\rangle
\ell_\lambda=\sum_{\lambda\in\Fcal_k}\langle
f_{j,\lambda}(\lambda) \ell_\lambda(\lambda),\ell_\lambda(\lambda)\rangle
\ell_\lambda
=\sum_{\lambda\in\Fcal_k}f_{j,\lambda}(\lambda)\ell_\lambda
\]
Therefore
\[
\begin{aligned}
 |\Pi_k(x,x)|&=\sum_{j=1}^N |s_j(x)|^2=\sum_{j=1}^N \langle
s_j(x),s_j(x)\rangle=
 \sum_{j=1}^N \Bigl\langle
\sum_{\lambda\in\Fcal_k}f_{j,\lambda}(\lambda)\ell_\lambda(x),
s_j(x)\Bigr\rangle\\
&=\sum_{\lambda\in\Fcal_k} \Bigl\langle \ell_\lambda(x),\sum_{j=1}^N
\overline{f_{j,\lambda}(\lambda)}s_j(x)\Bigr\rangle
=\sum_{\lambda\in\Fcal_k} \langle \ell_\lambda(x),\Phi_\lambda(x)\rangle,
\end{aligned}
\]
which gives \ref{ii}.
\end{proof}

We proceed with the proof of Theorem \ref{fekete_distribution}.

\begin{proof} [Proof of Theorem \ref{fekete_distribution}]
 Denote $\Omega:=B(z,\frac{R}{\sqrt{k}})$. By Lemma \ref{lem:dual} we have
 \[
 \begin{split}
  \#\Fcal_k\cap \Omega-\int_{\Omega}|\Pi_k(x,x)|=
  \sum_{\lambda\in \Fcal_k\cap \Omega} \int_X \langle
\ell_\lambda(x),\Phi_\lambda(x)\rangle-\int_\Omega \sum_{\lambda\in \Fcal_k}
\langle \ell_\lambda(x),\Phi_\lambda(x)\rangle=\\
=\int_{X\setminus \Omega} \sum_{\lambda\in\Fcal_k\cap \Omega} \langle
\ell_\lambda(x),\Phi_\lambda(x)\rangle-\int_\Omega \sum_{\lambda\in
F_k\cap(X\setminus \Omega)} \langle \ell_\lambda(x),\Phi_\lambda(x)\rangle=
A_1-A_2.
\end{split}
 \]
We first estimate $A_1$. We have 
\[
 |A_1|\le \int_{\Omega^c} \sum_{\lambda\in\Fcal_k\cap \Omega}
|\Pi_k(x,\lambda)|.
\]
By the sub-mean value property \eqref{InEqSubMean},
\[
 |\Pi_k(x,\lambda)| \lesssim \bigl(\frac \delta{\sqrt{k}}\bigr)^{-2n}
\int_{B(\lambda,\delta/\sqrt{k})} |\Pi_k(x,y)|\, dV(y).
\]
We take $\delta$ to be the separation constant of $\Fcal_k$ (Lemma \ref{LemFekSeparated}), then
\[
 \sum_{\lambda\in\Fcal_k\cap B(z,\frac{R-\delta}{\sqrt{k}})}
|\Pi_k(x,\lambda)|\lesssim k^n \int_\Omega |\Pi_k(x,y)|\, dV(y).
\]
Hence by part \ref{bk:ii} of Lemma~\ref{decay},
\[
 \int_{\Omega^c} \sum_{\lambda\in\Fcal_k\cap B(z,\frac{R-\delta}{\sqrt{k}})}
 |\Pi_k(x,\lambda)|\lesssim k^n \iint_{\Omega^c\times\Omega}
|\Pi_k(x,y)|\lesssim R^{2n-1}.
\]
On the other hand, the separation condition together with \eqref{shell_volume} imply
\[
 \# \Fcal_k \cap \Bigl( B\bigl(z,\frac{R}{\sqrt{k}}\bigr)\setminus
B\bigl(z,\frac {R-\delta}{\sqrt{k}}\bigr)
  \Bigr)  \lesssim R^{2n-1},
\]
and hence
\[
 \int_{\Omega^c} \sum_{\lambda\in\Fcal_k\cap
\bigl(B(z,\frac{R}{\sqrt{k}})\setminus
B(z,\frac {R-\delta}{\sqrt{k}})\bigr)
  }
 |\Pi_k(x,\lambda)|\lesssim 
 R^{2n-1}\sup_{\lambda}\int_X |\Pi_k(x,\lambda)|\, dV(x)\lesssim R^{2n-1},
\]
using part \ref{bk:i} of Lemma~\ref{decay}.
Combining the two estimates yields $|A_1|\lesssim R^{2n-1}$. In the same way,
we can also get the estimate $|A_2|\lesssim R^{2n-1}$.
Hence using \eqref{curvature_eigenvalues}
and \eqref{bergman_kernel_eigenvalues},
\begin{equation}\label{fek_pt_numer}
 \#\Fcal_k\cap \Omega =\int_{\Omega }
|\Pi_k(x,x)|+O(R^{2n-1}) = (1 + O(R^{-1})) \frac{k^n}{\pi^n
n!}\int_\Omega 
(i \partial \bar{\partial} \phi)^n.
\end{equation}
We also have from \eqref{EqBergmanBasis}  that
\begin{equation}\label{fek_pt_denom}
\# \Fcal_k = \operatorname{dim} \HoLk = \int_X |\Pi_k(x,x)|  =
(1+O(k^{-1}))\frac{k^n}{\pi^n n!}
\int_X (i \partial \bar{\partial} \phi)^n.
\end{equation}
Since we may assume $\frac{R}{\sqrt{k}} \leq\textrm{diam} (X)$, combining
\eqref{fek_pt_numer} with \eqref{fek_pt_denom} proves the theorem. 
\end{proof}

\subsection{}
The estimate \eqref{fek_pt_numer} obtained for the number of Fekete points
in a ball shows, in particular, that a Fekete array $\{\Fcal_k\}$ for the
positive line bundle has the critical density, 
\[
D^{-} (\{\Fcal_k\}) = D^{+} (\{\Fcal_k\}) = \frac{1}{\pi^n n!}\,.
\]
It is easy to check that the density of the perturbed array
$\{\Fcal_{(1\pm\varepsilon)k}\}$ will be equal to the critical value
multiplied by $(1\pm\eps)^n$. Combining this with Corollary
\ref{cor_pert_fek} shows that the density threshold in 
Corollary \ref{densities} is sharp.

\begin{corollary}\label{cor_dens_sharp}
Let $(L, \phi)$ be positive. Then
\begin{enumerate}
\addtolength{\itemsep}{4pt}
\item[{\rm (i)}]
For any $\eps > 0$ there is a sampling array $\Lambda$ with
$D^{+} (\Lambda) < \frac{1}{\pi^n n!} + \eps$.
\item[{\rm (ii)}]
For any $\eps > 0$ there is an interpolation array $\Lambda$ with 
$D^{-} (\Lambda) > \frac{1}{\pi^n n!} - \eps$.
\end{enumerate}
\end{corollary}

\subsection{}
Given two probability measures $\mu$ and $\nu$ on a metric space $X$, one 
defines the Kantorovich-Wasserstein distance $W$ between them as
\[
W(\mu,\nu)=\inf \left\{ \iint_{X \times X} \dist(x,y) \, d\rho(x,y)\right\}
\]
where the infimum is taken over all Borel probability measures $\rho$ on $X 
\times X$ with marginals $\rho(\cdot,X) = \mu$ and $\rho(X,\cdot) = \nu$. 
This metric plays a key role in transportation problems, see for instance
 \cite{Villani09}.

In our setting we have two probability measures, the first one is the 
Fekete measure $\mu_k$ defined in \eqref{Feketemeasure}, and the second one
is the measure
$(i \partial \bar{\partial} \phi)^n$ normalized to have total mass $1$,
which we denote by $\nu$.
It is known, see \cite{Blumlinger90} for instance, that on a
Riemannian manifold if
$\mu_k(B(x,r))\to\nu(B(x,r))$ for all balls, as
guaranteed by Theorem~\ref{fekete_distribution}, then  $\mu_k$
converges weakly to $\nu$ as $k\to\infty$, where the latter means
that $\int f d\mu_k \to \int f d\nu$ for any continuous function $f$ on $X$.

The Kantorovich-Wasserstein distance metrizes the weak convergence 
of measures. Here we prove Theorem~\ref{wass_estimate} which describes the 
rate of convergence in the Kantorovich-Wasserstein distance. 
For the proof it will be convenient to recall the dual formulation, see \cite[formula (6.3)]{Villani09} 
 \begin{equation}\label{wass-def-dual}
 W(\mu,\nu)=\sup \left\{\Bigl|\int_{X} f d(\mu-\nu)\Bigr|: f\in \lipone(X) 
\right\},
 \end{equation}
where $\lipone(X)$ is the collection of all functions $f$  on $X$ satisfying 
$|f(x)-f(y)|\le d(x,y)$.

\begin{proof} [Proof of Theorem~\ref{wass_estimate}]
To prove the lower bound for the Kantorovich-Wasserstein distance 
we consider the function $f_k(x)=\dist(x,\Fcal_k)$. Then clearly $f_k\in 
\lipone(X)$, and
moreover $f_k$ vanishes on $\Fcal_k$. Hence by \eqref{wass-def-dual},
\[
 W(\mu_k,\nu) \ge \Big| \int_X f_k (d\mu_k-d\nu)\Big| =\int_X f_k d\nu.
\]
The function $f_k$ is bounded below by $\delta > 0$ outside the balls
$B(\lambda,\delta)$, $\lambda\in\Fcal_k$, and so
\[
\int_X f_k d\nu \ge \delta \cdot
 \, \nu\bigl(X\setminus \bigcup_{x\in\Fcal_k}B(x,\delta)\bigr)\ge
 \delta (1- C \delta^{2n}\#\Fcal_k ).
\]
We choose $\delta = \delta(k)$ such that $C \delta^{2n}\#\Fcal_k=1/2$.
Since $\#\Fcal_k\simeq k^n$ by \eqref{h_dim_pos}, this implies
\[
 W(\mu_k,\nu) \gtrsim k^{-1/2}.
\]


For the upper estimate we will use the following alternative definition of the
Kantorovich-Wasserstein distance which is equivalent to the original:
\begin{equation}\label{alternative}
 W(\mu,\nu)=\inf_{\rho\in S}  \iint_{X \times X} \dist(x,y) \, 
|d\rho(x,y)|
\end{equation}
where the infimum is now taken over the set $S$ of all complex  measures $\rho$ on 
$X \times X$ with marginals $\rho(\cdot,X) = \mu$ and $\rho(X,\cdot) = \nu$.
In order to prove \eqref{alternative} recall the dual formulation \eqref{wass-def-dual}.
Now, for any complex measure $\rho$ with marginals $\mu$ and $\nu$ we have
\[
 \Bigl|\int_{X} f d(\mu-\nu)\Bigr|=\Bigl|\iint_{X\times X} (f(x)-f(y)) 
d\rho(x,y)\Bigr|\le \iint_{X\times X} \dist(x,y)|d\rho(x,y)|.
\]
Therefore 
\[
W(\mu,\nu)\le\inf_{\rho\in S}  \iint_{X \times X} \dist(x,y) \, 
|d\rho(x,y)|,
\]
the other inequality being trivial.

We will first prove that $W(\mu_k,\nu_k)\lesssim 1/\sqrt{k}$ where $\mu_k$ is the 
Fekete measure defined in \eqref{Feketemeasure} and $\nu_k$ is the 
probability measure defined as
\[
 d\nu_k(y):=\frac 1{N_k} |\Pi_k(y,y)|\, dV(y),
\]
where $N_k=\#\Fcal_k$. This is a probability measure because of 
\eqref{EqBergmanBasis}.
We choose a complex measure $\rho$ to get an upper bound for $W(\mu_k,\nu_k)$ 
as
\[
 d\rho(x,y):=\frac 1{N_k} \sum_{\lambda\in \Fcal_k} \delta_\lambda(x)\times 
\langle \ell_\lambda(y),\Phi_\lambda(y)\rangle\, dV(y)
\]
where $\Phi_\lambda$ are the sections defined in Lemma~\ref{lem:dual} and 
$\ell_\lambda$ are the Lagrange sections. Observe that 
Lemma~\ref{lem:dual}\ref{i} implies that the marginal $\rho(\cdot,X)=\mu_k$ and
Lemma~\ref{lem:dual}\ref{ii} that the marginal $\rho(X,\cdot)=\nu_k$. Thus
\[
 W(\mu_k,\nu_k)\le \iint_{X\times X}\dist(x,y)|d\rho(x,y)|=
 \frac 1{N_k}\sum_{\lambda\in \Fcal_k}\int_{X} \dist(\lambda,y) |\langle 
\ell_\lambda(y),\Phi_\lambda(y)\rangle|\, dV(y).
\]
We know that by the definition of Fekete points, the Lagrange sections are 
bounded and  $|\ell_\lambda(y)|\le 1$, see \eqref{EqLagrangeBound},  and moreover 
$|\Phi_\lambda(y)|=|\Pi_k(y,\lambda)|$ (Lemma~\ref{lem:dual}\ref{iii}). 
Therefore,
\[
  W(\mu_k,\nu_k)\le \frac 1{N_k}\sum_{\lambda\in \Fcal_k}\int_{X} 
\dist(\lambda,y)|\Pi_k(y,\lambda)|\, dV(y)\lesssim 1/\sqrt{k},
\]
where here we have used the estimates of Lemma~\ref{decay}\ref{bk:iii}.
Finally, if we denote by $\nu$ the measure $(i\ddbar\phi)^n$ divided by its total mass, we 
observe that $W(\nu_k,\nu)\lesssim 1/k$ because the total variation 
$\|\nu_k-\nu\| \lesssim 1/k$, by 
\eqref{curvature_eigenvalues} and \eqref{bergman_kernel_eigenvalues}, and since
the total variation controls the Kantorovich-Wasserstein distance, see 
\cite[Theorem~6.15]{Villani09}. We have thus 
proved that $W(\mu_k,\nu)\lesssim 1/\sqrt{k}$ as desired.
\end{proof}

%
%

\section{Simultaneously Sampling and Interpolation
arrays}\label{simultaneous}

\subsection{}
In this section we assume that $X$ is a \emph{projective} manifold, but
we work with a metric $\phi$ on the line bundle $L$ which is
only \emph{semi-positive}. We will show that, if there is
a point in $X$ where $\phi$ has a strictly positive curvature, then
the sections of high powers of the line bundle resemble closely the functions in the
Bargmann-Fock space. This observation will allow us to establish Theorem
\ref{no_riesz_bases}, showing that in this case there are no arrays 
which are simultaneously sampling and interpolation for $(L, \phi)$.
The non-existence of simultaneously sampling and interpolation 
sequences is a recent result in the classical Bargmann-Fock
space \cite{AscFeiKai11,GroMal11}.

Actually we could have replaced the assumption that $X$ is projective 
by the apparently weaker condition that $X$ is a K\"ahler manifold.
However, the solution of Siu \cite{Siu84} to the Grauert-Riemenschneieder 
conjecture  shows that, under the hypothesis that $L$ is semipositive with
a point where it has a strictly positive
curvature, the base manifold $X$ is Moishezon, and being also K\"ahler it is
automatically projective \cite{Moishezon66}.

The proof of Siu also shows that under the hypothesis of the theorem, $L$ is big
and thus there is a strictly positive singular metric $\phi_s$ on $L$ that  is
in $L^1_{\text{loc}}$ and 
smooth on all points of $X$ outside a proper analytic set $E$, 
see \cite[Theorem 2.3.30]{MaMa07}.

\subsection{}
We fix a point $x_0\in X\setminus E$ where the original
metric on $L$ had positive curvature. 
\begin{definition}\label{normalized}
We say that we have \emph{normalized coordinates} in a
neighborhood of $x_0\in X\setminus E$ if we have a coordinate chart that is
mapped to a neighborhood of $0$ in $\Cbb^n$ and a local holomorphic frame
$e_L(z)$ such that the following conditions hold:
\begin{itemize}
\item The curvature form of the line
bundle at $x_0$ is given by $\Theta(0) = \sum_{j=1}^n dz_j\wedge d\bar z_j$;
\item $h(0)=1$ and $\frac{\partial h}{\partial z_j}(0)=\frac{\partial^2 h}
{\partial z_j \partial
z_k}(0)=0$;
\end{itemize}
where above $h(z)=|e_L(z)|^2$,
and $\Theta(z) = -\ddbar \log h(z)$ is the curvature form.
\end{definition}
This can always be arranged if the curvature of $h$ is smooth and
positive at the point $x_0$, by choosing appropriate coordinates and a
convenient local frame. Observe that in normalized coordinates 
\begin{equation}\label{normal}
h(z)=e^{-|z|^2+o(|z|^2)}.
\end{equation}

We fix now a neighborhood
$B(0,\delta)$ of the origin at $\Cbb^n$ that is mapped by normal coordinates
to a neighborhood $U$ of $x_0$ in $X$. 
\begin{definition}
We define the sets $\Sigma_k\subset \Cbb^n$
as follows: $\sigma\in \Sigma_k$ if and only if $\sigma/\sqrt{k}$ is mapped by
the normal coordinates to a point in $\Lambda_K\cap U$. By definition
$\Sigma_k\subset B(0,\delta \sqrt{k})$.
\end{definition}
If $\Lambda_k$ is both an interpolation and sampling array, we will construct a
sequence 
$\Sigma\subset \Cbb^n$ such that it is both interpolation and sampling for the
Bargmann-Fock space.
\begin{definition} Given $p\in [1,\infty)$
The Bargmann-Fock space $\mathcal{BF}^p$
consists of entire functions such that 
\[
 \|f\|^p_p:=\int_{\Cbb^n} |f(z)|^p e^{-p|z|^2/2} dm(z) <+\infty.
\]
When $p=\infty$ the natural norm is 
\[
 \|f\|_\infty:=\sup_{\Cbb^n} |f(z)|e^{-|z|^2/2}.
\]
\end{definition}
A sequence $\Sigma$ is sampling for the Bargmann-Fock space $\mathcal{BF}^2$ if
and only if
\[
 \|f\|^2_2 \lesssim \sum_\sigma |f(\sigma)|^2 e^{-|\sigma|^2}\lesssim \|f\|^2_2
\]
and it is interpolation for $\mathcal{BF}^2$ if given any values $\{v_\sigma\}$ 
there is a function $f\in \mathcal{BF}^2 $ such that $f(\sigma)=v_\sigma$ and
with 
the estimate
\[
 \|f\|^2_2 \lesssim \sum_\sigma |v_\sigma|^2 e^{-|\sigma|^2},
\]
provided that the right hand side is finite.

It is known, see \cite{AscFeiKai11} and \cite{GroMal11}, that there do not
exist sequences that are simultaneously sampling and interpolation in
$\mathcal{BF}^2(\Cbb^n)$.

The key ingredient in the construction of $\Sigma$ is that the sections of high
powers of the (locally positive)
line bundle behave as functions in the Bargmann-Fock space when properly
rescaled. This is a well known phenomenon that can be illustrated by the fact
that the  Bergman kernel universally converges to the Bergman kernel of
the Bargmann-Fock space in normal
coordinates if rescaled properly, see \cite{BleShiZel00}. 
The next
theorem is another illustration of the same fact. In order to state it we need
to introduce the notion of weak limits of sequences. If we have a collection of 
separated sequences $\Sigma_k\subset \mathbb C^n$ with a uniform separation
constant for all $k$ and another separated sequence $\Sigma\subset \mathbb C^n$ 
we say that $\Sigma_k$ converges weakly to $\Sigma$ if the corresponding
measures
$\mu_k= \sum_{\sigma\in \Sigma_k} \delta_{\sigma_k}$  converge weakly to
$\sum_{\sigma\in \Sigma} \delta_\sigma$. This notion was used extensively by
Beurling in his study of sampling sequences in the Paley-Wiener space and it
will also be useful in our context.

\begin{thm}\label{thmfock}
Let $\Lambda_k$ be a separated sampling array for $L^k$ and let
$\Sigma$ be any weak limit of a partial subsequence $\Sigma_k$, then $\Sigma$ is
a sampling sequence for $\mathcal{BF}^2(\Cbb^n)$. 

Let $\Lambda_k$ be an interpolation array for $L^k$ and let
$\Sigma$ be any weak limit of a partial subsequence of $\Sigma_k$, then $\Sigma$
is an interpolation sequence for $\mathcal{BF}^2(\Cbb^n)$.
\end{thm}
\begin{proof}
 Let us start by the interpolation part. Assume that $\Sigma$
is the weak limit of a partial subsequences of
$\Sigma_k$ that, with an abuse of notation, will be still
denoted by $\Sigma_k$.
Let us take a sequence of values $\{v_{\sigma}\}_{\sigma\in \Sigma}$,
$v_\sigma\in\Cbb$, with $\sum_{\sigma\in \Sigma} |v_\sigma|^2
e^{-|\sigma|^2}<\infty$. We are going to construct a sequence of functions
$f_k\in \mathcal H(B(0,M_k))$, with $M_k\to\infty$ such that
\[
\sup_k \int_{|z|<M_k} |f_k(z)|^2 e^{-|z|^2}dm(z) <\infty,
\]
and for all $\sigma\in \Sigma$, $\lim_k f_k(\sigma)= v_\sigma$. Thus by a normal
family argument we conclude that there is an interpolating function $f\in
\mathcal{BF}^2$ with $f(\sigma)=v_\sigma$. Actually we may assume without loss
of
generality that, except for a finite number of points, $v_\sigma=0$. This is
harmless if 
 \[
\limsup_{k\to\infty} \int_{|z|<M_k} |f_k(z)|^2 e^{-|z|^2}dm(z) \le C  \sum_\sigma
|v_\sigma|^2 e^{-|\sigma|^2}
\]
with $C$ a constant independent of the number of non-zero terms.

Since we are assuming that the metric is smooth, and we are using
normalized coordinates, we can use Definition~\ref{normalized} and find an
increasing sequence $M_k$, $\lim M_k\to \infty$ (but with $M_k/\sqrt{k}\to 0$)
such that around $x_0$, $h(z)^k\simeq e^{-k|z|^2}$ for all
$|z|< M_k/\sqrt{k}$. 

Take some given values $v_\sigma$. We denote by $\Sigma'\subset \Sigma$ the
finite set of points $\sigma\in \Sigma$ such that $v_\sigma\ne 0$. For $k$ big
enough  $|\sigma/\sqrt{k}|<M_k$ for all $\sigma\in \Sigma'$. For those
$\sigma\in \Sigma'$ there is an associated $\lambda_\sigma^k\in \Lambda_k$ such
that $\sqrt{k}\lambda_\sigma^k\to \sigma$ because $\Sigma_k\to \Sigma$ weakly
(here we are identifying the points in $\Cbb^n$ and in $X$ by its coordinate
chart). Consider the interpolation problem with data $v_\sigma e_L^k$ at the
points $\lambda_{\sigma}^k$, $\sigma\in \Sigma'$. By hypothesis there is a
section $s\in H^0(L^k)$ such that $s_k(\lambda_{\sigma}^k)=v_\sigma
e_L^k(\lambda_{\sigma}^k)$ and
\[
\|s_k\|^2\le \frac{C}{k^n} \sum_{\sigma\in\Sigma'} |v_\sigma|^2 h(\lambda_{\sigma}^k)^k . 
\]
Near $x_0$ we may write $s_k(z)=g_k(z)e_L^k(z)$ and thus
\[
\begin{split}
&\int_{|z|\le M_k/\sqrt{k}} |g_k(z)|^2 e^{-k|z|^2}dm(z)\lesssim \\
&\lesssim\|s_k\|^2 \le 
\frac C{k^n} 
\sum _{\sigma\in\Sigma'} |v_\sigma|^2 h(\lambda_{\sigma}^k)^k \le\frac
C{k^n}\sum_{\sigma\in\Sigma'} |v_\sigma|^2 e^{-k|\lambda_{\sigma}^k|^2} .
\end{split}.
\]

The functions $f_k(z)=g_k(\sqrt{k} z)$ are holomorphic in $|z|<M_k$ and they
satisfy
\[
 \int_{|z|<M_k} |f_k(z)|^2 e^{-|z|^2}\le C \sum_{\sigma\in\Sigma'} |v_\sigma|^2
e^{-|\sqrt{k}\lambda_{\sigma}^k|^2}.
\]
If we let $k\to\infty$ in the right hand side of the inequality  we
obtain:
\[
\limsup_{k\to\infty} \int_{|z|<M_k} |f_k(z)|^2 e^{-|z|^2}\lesssim \sum_{\sigma\in\Sigma'}
|v_\sigma|^2
e^{-|\sigma|^2}.
\]
\end{proof}

\subsection{}
The sampling part of  Theorem~\ref{thmfock} is slightly more involved. We
need an
approximation lemma that in an informal way shows that one can approximate
locally functions in the Bargmann-Fock space by sections of $L^k$. More
precisely, we
will work with semipositive holomorphic line bundles $L$ over a projective
manifold $X$ that have some point where the metric on $L$ has strictly positive
curvature. As we mentioned before, such bundles are big line bundles and
therefore they admit a strictly positive
singular metric $\phi_s$ that is in $L^1_{loc}$ and it is smooth away from an
analytic exceptional set $E\subset X$, see \cite[Theorem 2.3.30]{MaMa07}.

\begin{lemma}\label{lemma_approx}
Let $L$ be a semipositive holomorphic line bundle over a projective manifold
$X$ with some point where the metric on $L$ has positive curvature. We fix a
point $x_0\in X$ where it has strictly positive curvature and that is not
contained in the exceptional analytic set $E$ 
and consider normal coordinates around it and its corresponding frame $e(z)$.
Given any function $f$ in the 
Bargmann-Fock space, and any big $M>0$, there is a $k_0\in\mathbb{N}$ such that
for all $k\ge k_0$ there are global holomorphic sections
$s_k(z)=f_k(z)e_k(z)$ of $L^k$ such that in the normalized coordinates around
$x_0$:
\[\int_{|z|<M/\sqrt{k}} |f(\sqrt{k}z)-f_k(z)|^2 e^{-k|z|^2}dz\lesssim \frac
1{M^2} \|f\|^2/k^n\]
and
\[
\int_{|z|>M/\sqrt{k}} |s_k|^2_\phi \lesssim \frac 1{M^2} \|f\|^2/k^n.
\]
In particular $\|s_k\|^2\simeq \|f\|^2/k^n$ for all $k\ge k_0$.
\end{lemma}
Thus, in a sense, $s_k$ are global sections that approximate $f$ around $x$.

This Lemma follows from the $L^2$, $\bar\partial$-estimates on line
bundles for singular metrics. This is a refinement of H\"ormander's theorem that
is due to Demailly-Nadel, see \cite{Berndtsson10} where
a nice exposition can be found. We will use the following theorem.

\begin{thm}[Demailly-Nadel] Let $X$ be a projective manifold. Let $L$ be a
holomorphic line bundle over $X$ which has a possibly singular metric $\phi_s$
whose curvature satisfies 
\[
 i\ddbar \phi_s \ge \varepsilon \omega,
\]
where $\omega$ is a K\"ahler form. Let $f$ be an $L$-valued $\dbar$-closed form
of bidegree $(n,1)$. Then there is a solution $u$ to the equation $\dbar u=f$
satisfying
\[
 \|u\|_{\omega,\phi_s}^2\lesssim  \int_X |f|_{\ddbar \phi_s}^2 e^{-\phi_s}.
\] 
\end{thm}
In this statement $|f(x)|_{\ddbar\phi_s}$ is the pointwise norm on $(n,1)$ forms
induced by
the singular Hermitian metric in $X$.
In particular if we have the estimate $i\ddbar
\phi_s \ge M \omega$ in the support of $f$, then 
\begin{equation}
 \|u\|_{\omega,\phi_s}^2\lesssim
\frac 1{M} \|f\|_{\omega,\phi_s}^2.
\end{equation}
We prove now the approximation lemma.
\begin{proof}[Proof of Lemma \ref{lemma_approx}]
Let $\chi$ be a cutoff function supported in a ball of radius $M$ centered at
the origin and equal to $1$ in $B(0,M/2)$. We take $M$ so big that  $|\nabla
\chi |\le 4/M.$ We put $\chi_k(z)=\chi(z\sqrt{k})$. We define in normal
coordinates $g_k(z)=f(\sqrt{k}z)\chi_k(z)e_k(z)$. The section $g_k$ (extended
by $0$ outside a neighborhood of $x_0$) defines a global (non-holomorphic)
section with the required properties. To make it holomorphic we must correct it
with the equation $\dbar u_k = \dbar g_k$ and define $s_k=g_k-u_k$. We need to
make sure that the correction $u_k$ is globally small. 

One technical difficulty arises: the H\"ormander estimates for the
$\dbar$-equation deal with $(n,1)$-forms rather than $(0,1)$-forms. We can
always twist the line bundle $L$ with the canonical bundle to shift from 
$(0,1)$-forms to $(n,1)$-forms. In this case this is delicate because while
twisting the bundle we could lose its positivity since $L$ is only semipositive
and there is no maneuvering room. For this purpose we will need to change the
metric on $L$ to make it strictly positive while preserving the estimates in
the original metric. This can be achieved by averaging the original metric
$\phi$ on $L$ with the metric $\phi_s$ that is singular and strictly positive
on $L$. That is the reason we need to work with the more
sophisticated Demailly-Nadel estimates on singular metrics rather than the
H\"ormander estimates. More precisely, let us define a new metric
$\widetilde\phi_k$ on $L^k$ as follows:
\begin{equation}\label{metric}
  \widetilde\phi_k= (k-N) \phi + N \phi_s - C,
\end{equation}
where $N$ and $C$ are big constants, that do not depend on $k$, to be chosen.
This is a well defined singular metric on $L^k$ since $\widetilde\phi_k= k\phi
+N (\phi_s-\phi)-C$ and the difference of two metrics $\phi_s-\phi$ is a well
defined function on $X$. 

The bundle $L^k$ can be expressed as $L^k=K_X\otimes F_k$, where $K_X$ is the
canonical line bundle. If we endow $L^k$ with the metric $\widetilde \phi_k$ and
$K_X$ with the metric inherited from the Hermitian metric on $X$, the curvature
of $F_k$ is
\[
c(F_k)=c(\widetilde \phi_k)-c(K_X)= (k-N)c(\phi) + N c(\phi_s)
-c(K_X) \ge N\varepsilon \omega - c(K_X),
\]
if $k>N$ and thus it has positive curvature taking $N$ big enough,
where by $c(\cdot)$ we denote here the curvature form of the corresponding
line bundle or metric specified. In fact  on the support of $\dbar g_k$ the
curvature satisfies $c(F_k)\gtrsim k \omega$.

The metric $\phi_s$ is bounded above because it is in $L^1_{loc}$ and it is
plurisubharmonic. Thus we can take the constant $C$ big enough in \eqref{metric}
in such a way that $\widetilde\phi_k\le k\phi$.

The $L^2$ norm of $\dbar g_k$ with the metric $\widetilde \phi_k$ is comparable
to the $L^2$ norm with respect to the metric $k\phi$ because  $\phi_s$ is 
smooth on the support of $\dbar g_k$, thus its norm is bounded by $k^{1-n}M^{-2}
\|f\|^2$. If we solve the $\dbar$ equation using the estimates provided by the
Demailly Nadel theorem with data that is a $(n,1)$-form with values in $F_k$ we
get a solution $u_k$ to $\dbar u_k =\dbar g_k$ ($u_k$ is a global $(n,0)$-form
with values in $F_k$ or equivalently a global section of $L^k$) with $L^2$ size
controlled by a constant times $k^{-n}M^{-2} \|f\|^2$ as desired. A priori the
norm control of $u_k$ is with respect to $\widetilde\phi_k$ but as
$\widetilde\phi_k\le k\phi$ we get the desired result.
\end{proof}

We proceed now to prove the sampling part of Theorem~\ref{thmfock}. Given
any function
$f$ in the
Fock space we take a large $M>0$ so that
\[
\int_{|z|>M} |f|^2e^{-|z|^2} \le 0.1 \|f\|^2.
\]
We can construct a sequence of sections $s_k$ such that
the conclusions of the approximation lemma hold.
For such $s_k$ the sampling property of $\Lambda_k$ can be applied and we have 
\[
 \|s_k\|^2\lesssim \frac 1{k^n}\sum_{\lambda\in \Lambda_k}
|f_k(\lambda)|^2e^{-k\phi(\lambda)}.
\]
Since all the $f_k$ have $L^2$ norm very small outside the region parametrized
by $|z|<M/\sqrt{k}$ which we denote by $U_k$ the mean value property implies
that
\[
\|s_k\|^2\lesssim \frac 1{k^n}\sum_{\lambda\in \Lambda_k\cap U_k}
|f_k(\lambda)|^2e^{-k\phi(\lambda)}.
\]
We recall that $k^n\|s_k\|^2\simeq \|f\|^2$, and taking weak limits of
$\Sigma_k$, implies that
\[
 \|f\|^2 \lesssim \sum_{|\sigma|\le M} |f(\sigma)|^2 e^{-|\sigma|^2}.\qed
\]

\section{The one-dimensional case}\label{onedim}

In this section we return to discuss positive line bundles, and
focus on the case when $\dim(X)=1$, i.e. we are dealing with a compact Riemann 
surface. In this case we have a more precise result, namely a 
full characterization of the interpolation and
sampling arrays given by Theorem \ref{iff_conditions} above.

\subsection{}
The sampling part of Theorem \ref{iff_conditions} can be reformulated as follows.

\begin{thm}\label{sufsampling}
Let  $\Lambda$ be a separated array and let $L$ be a holomorphic line bundle
over a compact Riemann surface $X$ endowed with a smooth positive metric
$\phi$. Then $\Lambda$ is a sampling array for the line bundle if and only if there
is an $\varepsilon>0$, $r>0$ and $k_0>0$ such that for all
$k\ge k_0$, 
\begin{equation}\label{updensity} \frac{\# (\Lambda_k \cap B(x,
r/\sqrt{k}))}{\int_{B(x, r/\sqrt{k})}ik\ddbar \phi } > \frac 1\pi
+\varepsilon
\qquad\forall x\in X, 
\end{equation} 
 
\end{thm}

Remark that the metric in $X$ used to define the balls in the inequality 
\eqref{updensity} is irrelevant, since the density inequality is invariant
under change of metric. We will prove
this invariance in an arbitrary dimension. Assume that we have two different
metrics that induce two distances $d_1$ and $d_2$ and two volumes $V_1$ and
$V_2$. Suppose that \eqref{updensity}
holds for the first metric. Denote by $\mu_k:=\frac 1{k^n}\sum_{\lambda \in
\Lambda_k}
\delta_\lambda$ and $\nu:= (i\ddbar \phi)^n$. The hypothesis \eqref{updensity}
(in dimension $n$) can we written as 
\begin{equation}\label{ineqmeasures}
 \int_{B_1(y,r/\sqrt{k})} d\mu_k(x) \ge \bigl(\frac 1{\pi^n n!}
+\varepsilon\bigr)
\int_{B_1(y,r/\sqrt{k})} d\nu.
 \end{equation}
We need some notation to check that  \eqref{ineqmeasures} is invariant under
change of metrics. Denote by 
\[
 \widetilde f_r(z):= \frac 1{\lambda_0(r/\sqrt{k})} \int_{B_1(z,r/\sqrt{k})}
f(y) dV_1(y)=\frac 1{\lambda_0(r/\sqrt{k})} \int_{X} f(y)
\1_{B_1(y,r/\sqrt{k})}(z) 
dV_1(y),
\]
where $\lambda_0(r)$ denotes as in \cite{Blumlinger90} the volume of a Euclidean
ball of radius $r$ in $\Rbb^{2n}$. Thus
\[
 \int_X \widetilde f_r\,d\mu_k=\frac 1{\lambda_0(r/\sqrt{k})} \int_X
f(y)\mu_k(B_1(y,r/\sqrt{k}))dV_1(y).
\]
For any $f\ge 0$, we have by \eqref{ineqmeasures}
\[
 \int_X \widetilde f_r  \, d\mu_k\ge \bigl(\frac 1{\pi^n n!}
+\varepsilon\bigr) \int_X \widetilde f_r \, d\nu.
\]
We choose $f:=\1_{B_2(x,R/\sqrt{k})}$, then
\[
 \1_{B_2(x,(R-cr)/\sqrt{k})}\le f_r \le \1_{B_2(x,(R+cr)/\sqrt{k})},
\]
where 
\[
 f_r(z):=\frac 1{V_1(B_1(z,r/\sqrt{k}))} \int_{B_1(z,r/\sqrt{k})} f(y) dV_1(y).
\]
The following inequalities are now elementary:
\[
\begin{split}
\mu_k(B_2(x,(R+cr)/\sqrt{k}))\ge \int_X f_r d\mu_k=\int_X \widetilde f_r d\mu_k
+\int_X (f_r-\widetilde f_r)d\mu_k\ge \\
\bigl(\frac 1{\pi^n n!}
+\varepsilon\bigr) \int_X \widetilde f_r d\nu + \int_X (f_r-\widetilde
f_r)d\mu_k=\\
\bigl(\frac 1{\pi^n n!} +\varepsilon\bigr) \int_X f_r d\nu +
\bigl(\frac 1{\pi^n n!} +\varepsilon\bigr) \int_X (\widetilde f_r -f_r)d\nu +
\int_X (f_r-\widetilde f_r)d\mu_k\ge\\
\bigl(\frac 1{\pi^n n!} +\varepsilon\bigr)\nu(B_2(x,(R-cr)/\sqrt{k}))+ 
\bigl(\frac 1{\pi^n n!} +\varepsilon\bigr)\int_X(\widetilde f_r -f_r)d\nu +
\int_X (f_r-\widetilde f_r)d\mu_k.
\end{split}
\]
We aim to prove that 
\begin{equation}\label{des}
\mu_k(B_2(x,(R+cr)/\sqrt{k}))\ge\Bigl(\frac 1{\pi^n n!}
+\frac{\varepsilon}2 \Bigr)\nu\bigl(B_2(x,(R+cr)/\sqrt{k})\bigr).
\end{equation}
Clearly if $R$ is big enough ($R>>cr$), then by \eqref{shell_volume}:
\[
 \nu\Bigl(\Bigl\{ y: \frac{R-rc}{\sqrt{k}}\le  d_2(y,x)\le
\frac{R+cr}{\sqrt{k}}\Bigr\}\Bigr)
 \le \frac{\varepsilon}4 \nu\bigl(B_2(x,(R-rc)/\sqrt{k})\bigr).
\]
We still need to prove that the terms $\bigl(\frac 1{\pi^n n!}
+\varepsilon\bigr)
\int_X(\widetilde f_r -f_r)d\nu +
\int_X (f_r-\widetilde f_r)d\mu_k$ are negligible when compared to
$\nu\bigl(B_2(x,(R-rc)/\sqrt{k})\bigr)\simeq R^{2n}/k^n$ as $k\to\infty$.

Observe that $|f_r-\widetilde f_r|\le K_1(r/\sqrt{k}) f_r$, where
$K_1(s)=\sup_X |1-V_1\bigl(B_1(x,s)\bigr)/\lambda_0(s)|$.
The distortion function $K_1(s)=O(s^2)$ \cite[Lemma~2]{Blumlinger90}, and thus
\[
 \Bigl| \int_X (f_r-\widetilde f_r)d\nu\Bigr| \le K_1\Bigl(\frac
r{\sqrt{k}}\Bigr)\int_X f_rd\nu\le
 K_1\Bigl(\frac r{\sqrt{k}}\Bigr) \nu
\Bigl(B\Bigl(x,\frac{R+cr}{\sqrt{k}}\Bigr)\Bigr)\lesssim \frac 1{k^{n+1}}.
\]
We assume that $\Lambda$ is separated, thus
$\mu_k\bigl(\bigl(x,\frac{R+cr}{\sqrt{k}}\bigr)\bigr)\lesssim 
\frac{R^{2n}}{k^n}$, and therefore 
\[
 \Bigl| \int_X (f_r-\widetilde f_r)d\mu_k\Bigr| \le K_1\Bigl(\frac
r{\sqrt{k}}\Bigr)\int_X f_rd\nu\le
 K_1\Bigl(\frac r{\sqrt{k}}\Bigr) \mu_k
\Bigl(B\Bigl(x,\frac{R+cr}{\sqrt{k}}\Bigr)\Bigr)\lesssim \frac 1{k^{n+1}},
\]
and if we take $k$ big enough, we have proved \eqref{des} and the invariance 
of the density condition under changes of metric follows.

\subsection{}
We proceed now to the proof of Theorem~\ref{sufsampling}. We start by proving
that under the density hypothesis \eqref{updensity} the array $\Lambda$ is
sampling. We will initially prove that the array is $L^\infty$-sampling.

\begin{definition}
We say that a separated array $\Lambda=\{\Lambda_k\}$ is an
$L^\infty$-\emph{sampling array} if there is $k_0$ and 
a constant $0<C<\infty$ such that, for each
$k \geq k_0$ and any section $s\in \HoLk$ we have
\[
\sup_{x \in X} \big | s(x)\big| \leq 
C \sup_{\lambda\in \Lambda_k} |s(\lambda)|.
\]
\end{definition}

 If this were not true then for infinitely many $k$'s there will be
$s_k\in H^0(L^k)$ and points $x_k\in X$ such that
\[
 \sup_X |s_k|=|s_k(x_k)|=1,
\]
and 
\[
 \sup_{\lambda\in\Lambda_k} |s_k(\lambda)|=o(1).
\]
We take normal coordinates around $x_k$, see Definition~\ref{normalized}, and we
consider as before arrays
$\Lambda_k\subset X$ and $\Sigma_k\subset B(0,M_k)$ the dilated sequences in
$\Cbb$. Since $\Sigma_k$ are separated there is a subsequence converging weakly
to
$\Sigma$ that for simplicity we keep denoting by $\Sigma_k$.  The hypothesis
implies that
\[
 \frac{\# \Sigma \cap B(y,r_0)}{r_0^2}\ge (\frac 1\pi +\varepsilon),
\]
the balls $B(y,r_0)$ are standard balls in $\Cbb$ because we may choose a metric
in $X$ such that when rescaling around $x$ by the normal coordinates it
converges to
the Euclidean metric in $\Cbb$. By a theorem of Seip and Wallsten, see 
\cite[Theorem~1.1]{SeiWal92} $\Sigma$ is sampling for
the
space of functions $\mathcal{BF}^\infty$ consisting of entire functions such
that $\sup |f|e^{-|z|^2}<+\infty$. On the other hand we may extract a converging
subsequence of functions $f_k$ that represent the sections $s_k$ in normal
coordinates to $f\in \mathcal{BF}^\infty$ such that $|f(0)|=1$ and
$f|_{\Sigma}=0$ and this is a contradiction with the fact that $\Sigma$ is
sampling for $\mathcal{BF}^\infty$.

Once we know that $\Lambda$ is $L^\infty$ sampling it is possible to argue as
with the Fekete points that $\{\Lambda_{(1+\varepsilon) k}\}$ is $L^2$ sampling.

\begin{proposition}\label{saltsampling}
 If $\Lambda=\{\Lambda_k\}$ is $L^\infty$ sampling then
$\{\Lambda_{(1+\varepsilon) k}\}$ is $L^2$
sampling.
\end{proposition}
\begin{proof}
 We know by hypothesis that for any $s\in H^0(L^k)$, $\sup_X |s|\le C
\sup_{\Lambda_k} |s(\lambda_k)|$. In this case it is elementary to check that
$\{\Lambda_{(1+\varepsilon) k}\}$ is also $L^\infty$ sampling. For any  $s\in
H^0(L^k)$, and $y\in X$ we define the
section 
\[
p_y(x)=s(x)\otimes
\left[\frac{\Phi_y^{(\varepsilon/2)k}(x)}{|\Pi_{(\varepsilon/2)k}(y,y)|}\right]
^2\in H^0(L^{(1+\varepsilon)k})
\]
Let us take now $y\in X$ to be a point where $|s|$ attains its maximum. Then
\begin{equation}\label{acotada}
 \sup_X |s|= |s(y)|=|p_y(y)|\le C \sup_{\Lambda_{(1+\varepsilon) k}}
|p_y(\lambda)|\le
C\sup_{\Lambda_{(1+\varepsilon) k}} |s(\lambda)|.
\end{equation}

Moreover for any $z\in X$, since $\Lambda$ is sampling, 
\[
\begin{split}|s(z)|=|p_z(z)| \lesssim \sup_{\Lambda_{(1+\varepsilon) k}}
|s(\lambda)|
\left|\frac{\Phi_z^{(\varepsilon/2)
k}(\lambda)}{|\Pi_{(\varepsilon/2)k}(z,z)|}\right|^2\le
\sum_{\Lambda_{(1+\varepsilon) k}}|s(\lambda)|
\left|\frac{\Phi_z^{(\varepsilon/2)
k}(\lambda)}{|\Pi_{(\varepsilon/2)k}(z,z)|}\right|^2=\\
\sum_{\Lambda_{(1+\varepsilon) k}}|s(\lambda)|
\left|\frac{\Pi_{(\varepsilon/2)
k}(z,\lambda)}{|\Pi_{(\varepsilon/2)k}(z,z)|}\right|^2.
\end{split}
\]
Recall that $|\Pi_{(\varepsilon/2)k}(z,z)|\simeq \eps k$. Thus if we integrate both
sides, we get
\begin{equation}\label{integrable}
 \int_X |s(z)| \lesssim \frac1{\eps k} \sum_{\Lambda_{(1+\varepsilon)
k}} |s(\lambda)|
\end{equation}
If we interpolate between \eqref{acotada} and \eqref{integrable} we obtain 
\[
 \int_X |s(z)|^2 \lesssim \frac1{\eps k} \sum_{\Lambda_{(1+\varepsilon)
k}} |s(\lambda)|^2,
\]
as stated.
\end{proof}
Finally, since the hypothesis of Theorem~\ref{sufsampling} is an open condition
we can conclude that actually $\{\Lambda_{(1-\varepsilon) k}\}$ is
$L^\infty$-sampling
and therefore $\Lambda$ is $L^2$-sampling.

We turn now to the necessity of the density condition. We assume that  $\Lambda$
is a sampling array. We know already by
Corollary~\ref{densities} that the density of $\Lambda$ is bigger or equal
than a critical level. We need a strict inequality. We prove now that
if $\Lambda$ is a sampling array then there is
a $\varepsilon>0$ such that $\{\Lambda_{(1-\varepsilon)k}\}$ is still an
$L^2$-sampling array. 

 We know by Theorem~\ref{thmfock} than any weak limit $\Sigma\in W(\Lambda)$ is
a sampling sequence in 
$\mathcal{BF}^2(\Cbb)$. Thus by the description of sampling
sequences for such spaces obtained in \cite{SeiWal92}, the lower Beurling
density $D^-(\Sigma)>1$. We will prove that under this circumstances there is a
$\varepsilon>0$ such that $\{\Lambda_{(1-2\varepsilon)k}\}$ is
$L^\infty$-sampling.
We argue by contradiction. Suppose not, then, for any $n$ there are
sections $s_k\in H^0(L^k)$ such that $\|s_k\|_\infty=1$ and
$\|s|_{\Lambda_{(1-1/n)k}}\|_\infty=o(1)$ when $k$ is very big.
If we fix $n$ and by passing to a subsequence in normal coordinates around the
points $x_k$ where $|s_k|$ takes its maximum value, we  construct functions
$f_n\in \mathcal{BF}^\infty(\Cbb)$ of norm one such that $f_n(0)=1$ and
$f_n|_{\Sigma_n}\equiv 0$, where $\Sigma_n$ is a weak limit of a subsequence of 
$\Lambda_{(1-1/n)k}$ as $k\to \infty$ in normal coordinates scaled
appropriately. We take  another subsequence of the functions $f_n$ and of
the separated sequences $\Sigma_n$ in such a way that $\Sigma_n$ converge
weakly to $\Sigma$, $f_n\to f$ and $f\in \mathcal{BF}^\infty(\Cbb)$ of norm
one, $f(0)=1$, $f|_\Sigma\equiv 0$ and $\Sigma\in W(\Lambda)$. This is a
contradiction since $\Sigma$ has $D^-(\Sigma)>1$.

We have proved that $\{\Lambda_{(1-2\varepsilon)k}\}$ is
$L^\infty$-sampling. We finish the proof by observing that
by Proposition~\ref{saltsampling} this implies that 
$\{\Lambda_{(1-\varepsilon)k}\}$ is $L^2$-sampling.
\qed

\subsection{}
We provide now a characterization for the interpolation arrays.
\begin{thm}\label{interpolationthm}
 Let $\Lambda$ be a separated array and let $L$ be a
holomorphic line bundle with a smooth positive metric $\phi$ over a compact Riemann
surface $X$. Then $\Lambda$ is an interpolation array for the line bundle if
and only if there is an $\varepsilon>0$, $r>0$ and $k_0$ such that for all $k\ge
k_0$,
\begin{equation}\label{density}
\frac{\# (\Lambda_k \cap B(x,
r/\sqrt{k}))}{\int_{B(x, r/\sqrt{k})}ik\ddbar \phi
} < \frac 1\pi -\varepsilon  \qquad\forall x\in X,
\end{equation}
\end{thm}

Remark that the density condition \eqref{density} is invariant under change 
of metric, which can be shown in a simiar way as we did above for the condition
\eqref{updensity}.
We will first check that condition \eqref{density} implies that
$\Lambda$ is interpolating. We start by the following reduction.

 \begin{proposition}\label{weakint} Let $\Lambda$ be separated.
  If there is a $C>0$  such that for every 
  $k\ge k_0$ and every $\lambda\in\Lambda_k$ there is a section
  $s_\lambda\in H^0(L^k)$ with 
  \begin{enumerate}
\renewcommand\theenumi{(\roman{enumi})}
\renewcommand{\labelenumi}{\theenumi}

   \item\label{uno} $|s_\lambda(\lambda)|=1$;
   \item\label{dos} $\sup_\lambda \sum_{\lambda'\ne\lambda}
|s_\lambda(\lambda')|< 1/2$;
   \item\label{dosbis} $\sup_{\lambda'} \sum_{\lambda\ne\lambda'}
          |s_\lambda(\lambda')|< 1/2$;
   \item\label{tres} $\|\sum c_\lambda s_\lambda\|_2^2 \le Ck^{-1} \sum
           |c_\lambda|^2;$ 
  \end{enumerate}
  then $\Lambda$ is an interpolation array.
\end{proposition}
\begin{proof}
Let $\ell^2(\Lambda_k)$ be endowed with the norm
$\|v\|^2:=k^{-1}\sum_{\Lambda_k} |v_\lambda|^2$.
We consider the following two operators.
The first is the restriction operator $R: H^0(L^k)\to \ell^2(\Lambda_k)$ defined as
$R(s)=\{s(\lambda)\}$. It is bounded from $H^0(L^k)$ endowed with the $L^2$ norm
by the Plancherel-P\'olya inequality (Lemma \ref{LemmaPlancherelPolya})
since $\Lambda$ is separated and
its norm $\|R\|$ depends only on the separation constant of $\Lambda$.

The second operator is $E:\ell^2(\Lambda_k) \to H^0(L^k)$ defined as
$E(\{v_\lambda\})=\sum \dotprod{v_\lambda}{s_\lambda(\lambda)} s_\lambda(x)$.
It is bounded clearly by properties \ref{uno} and \ref{tres}. If we prove that $RE: \ell^2\to \ell^2$ is invertible with
the norm of the inverse $\|(RE)^{-1}\|\le C$ bounded independently of $k$, then
clearly $\Lambda$ is interpolating, because any values $\{v_\lambda\}$ are
attained by the section $s= E (RE)^{-1} (\{v_\lambda\})$ with size control.

But the conditions  \ref{uno},\ref{dos},\ref{dosbis} imply that the operator
$RE-\operatorname{Id}: \ell^2(\Lambda_k) \to \ell^2(\Lambda_k)$
by Schur's Lemma has norm bounded by $1/2$. Thus $RE$ is invertible. 
\end{proof}

To finish the proof of the sufficiency of \eqref{density} we are going to
construct the sections as in the
Proposition~\ref{weakint}. Around any given $\lambda\in \Lambda_k$ we can
consider normal
coordinates. Since by hypothesis the density is small the corresponding sequence
$\Sigma_k$ is an interpolation sequence for the $\mathcal{BF}^2$ space in
$\Cbb$.
Actually since the separation constant is uniform and the density is
uniform then by a theorem of Seip and Walsten, see 
\cite[Theorem~1.2]{SeiWal92}, the constants of interpolation
for all the sequences $\Sigma_k$ around any point $\lambda\in \Lambda_k$ will
be uniformly bounded, for $k\ge k_0$. Thus we can construct  functions
$f_\lambda^k$ such that $|f_\lambda^k(0)|=1$,
$\|f_\lambda^k\|\le C$ and $f_\lambda^k(\sigma)=0$ for all $\sigma\in
\Sigma_k\setminus 0$. Now we can construct a global section $g_\lambda\in
H^0(L^k)$ such that near $\lambda$, $g_\lambda(z)$ is very close to
$f_\lambda(z)^k e_L^k(z)$, where $e_L^k(z)$ is the local frame around $\lambda$
used for the normal coordinates. 

In order to do this we define $g_\lambda = \chi_{\lambda,k}(z) f_\lambda^k(z)
e^k(z) + u$, where $\chi_{\lambda,k}$ is a cutoff function around $\lambda$
such that $g_\lambda(z)=0$ if $d(z,\lambda)>2C/\sqrt{k}$ and $g_\lambda(z)=1$
if $d(z,\lambda)<C/\sqrt{k}$   and $u$ is the solution to the equation $\dbar
u= \dbar \chi_{\lambda,k}f_\lambda^k(z) e^k(z)$ provided by the H\"ormander
theorem. This theorem ensures that $\|u\|^2\le \varepsilon$, 
provided that the cutoff constant $C$ is big enough.

This is not enough if we want the decay needed in the Proposition~\ref{weakint},
in
particular in the items \ref{dos},\ref{dosbis} and \ref{tres}. We are again going to use the extra freedom
that we have because the hypothesis is an open condition. We could have taken
$f_\lambda^k$ such that $\int |f^k_\lambda|^2 e^{-(1-\varepsilon)|z|^2}<+\infty$
and in this
case we could have constructed $g_\lambda\in H^0(L^{(1-\varepsilon)k})$ such
that 
\[
 |g_\lambda(\lambda)|=1,\qquad \|g_\lambda\|^2\le C/k,\qquad
k^{-1} \sum_{\lambda'\ne \lambda} |g_\lambda(\lambda')|^2 \le
\varepsilon.
\]
and we can take in the construction $\varepsilon>0$ as small as we want without
affecting the $K$.

We define $s_\lambda(x)=g_\lambda(x) \otimes
\left[\frac{\Phi_{\lambda}^{\varepsilon
k/2}(x)}{|\Pi_{(\varepsilon/2)k}(\lambda,\lambda)|}\right]^2$ and
using \eqref{EstimateOffDiagonal} it is easy to check that 
\[\sup_X |\sum_{\Lambda_k} c_\lambda s_\lambda(x)| \lesssim \sup_{\Lambda_k}
|c_\lambda| \quad \text{ and } \quad
\int_X |\sum_{\Lambda_k} c_\lambda s_\lambda(x)| \lesssim
k^{-1} \sum_{\Lambda_k} |c_\lambda|.
\] 
Thus by interpolation we get $\|\sum c_\lambda s_\lambda\|_2^2 \le C k^{-1}
\sum         |c_\lambda|^2$ which gives \ref{tres}. Finally \ref{dos},\ref{dosbis}
can be checked in a similar way.
\qed

\subsection{}
We turn now to the neccesity of the density condition \eqref{density}. We need
to check that the density condition that we
proved that was necessary in Corollary~\ref{densities} is actually a strict
density condition. 
As a technical tool to prove the necessity of the strict inequality we need to
work with $L^1$ interpolating arrays. The definition is the following:

\begin{definition}
We say that a separated array $\Lambda=\{\Lambda_k\}$ is an
$L^1$-\emph{interpolation array} if there is $k_0$ and
a constant $0<C<\infty$ such that, for each $k \geq k_0$ and any
set of vectors $\{v_\lambda\}_{\lambda\in \Lambda_k}$ (each $v_\lambda$ is an
element of the fiber of $\lambda$ in $\Lcal^k$) there is a section
$s\in \HoLk$ such that
\begin{equation*}
s(\lambda)=v_\lambda, \quad \lambda\in \Lambda_k,
\end{equation*}
and
\begin{equation}\label{eq_l1_int}
\int_X \big | s(x) \big | \leq C k^{-1} \sum_{\lambda
\in\Lambda_k} |v_\lambda|.
\end{equation}
\end{definition}

On each level $k\ge k_0$, the best constant $C_k$ such that
\eqref{eq_l1_int} holds for all $s\in \HoLk$ that interpolate the 
prescribed values, is called the constant
of interpolation at level $k$. Of course $\Lambda$ is an interpolation array
if all the constants $\{C_k\}$ are uniformly bounded.
There is an alternative way of computing $C_k$ by duality.

\begin{proposition}\label{duality}
The constant of $L^1$ interpolation at level $k$ is comparable to the smallest
constant $A_k$ such that  
\[
 \sup_{x\in X} k^{-1} \Big|  \sum_{\Lambda_k} 
 \dotprod{a_\lambda}{\Pi_k(x, \lambda)}
 \Big| \le A_k \sup_{\Lambda_k} |a_\lambda|,
\]
where
$\{a_\lambda\}_{\lambda\in \Lambda_k}$ are arbitrary elements on the fiber of
$\lambda$ in $\Lcal^k$.
\end{proposition}

\begin{proof}
 This is standard and follows from the fact that the Bergman kernel
decays very fast away from the diagonal \eqref{EstimateOffDiagonal}.
Thus the Bergman projection from
sections of $L^k$ endowed with the $L^p$ norm to holomorphic sections endowed
with the  $L^p$ norm is bounded for all $p\in[1,\infty]$, and the dual space of 
 $H^0(L^k)$ with the $L^1$ norm is the space $H^0(L^k)$ endowed with the
supremum norm.
\end{proof}
It will be convenient to compare interpolating arrays in $L^1$ and in $L^2$
and we will use the following proposition
\begin{proposition}\label{saltinterp}
If $\Lambda=\{\Lambda_k\}$ is an $L^1$ interpolation array then
$\{\Lambda_{(1-\varepsilon) k}\}$ is an $L^2$ interpolation array.
\end{proposition}
\begin{proof}
 If $\Lambda$ is an $L^1$ interpolation array then for each $\lambda\in
\Lambda_{(1-2\varepsilon)k}$
we can build a ``Lagrange type'' section $s_\lambda\in
H^0(L^{(1-2\varepsilon)k})$ such that 
$|s_\lambda(\lambda)|=1$, $|s_\lambda(\lambda')|=0$ for all
$\lambda'\in \Lambda_{(1-2\varepsilon)k}\setminus\{\lambda\}$,
and $\|s_\lambda\|_{L^1} \le C
/k$. Then by the sub-mean value property \eqref{InEqSubMean} 
we obtain $\sup_X |s_\lambda(x)|\le C k \|s_\lambda\|_{L^1}\le
C$. Thus we can use the same argument as in Theorem \ref{LemFekSamInt} and we
prove
that $\{\Lambda_{(1-\varepsilon) k}\}$ is an $L^2$-interpolation array.
\end{proof}

The proof of strict inequality \eqref{density} follows once
we establish the following

\begin{proposition}
Assume that $\dim(X)=1$. Let $\Lambda$ be an $L^2$-interpolating array. There
is $\varepsilon>0$ such that $\{\Lambda_{(1+\varepsilon)k}\}$ is
$L^2$-interpolating. 
\end{proposition}
\begin{proof}
 We know by Theorem~\ref{thmfock} than any weak limit $\Sigma\in W(\Lambda)$ is
an interpolating sequence in 
$\mathcal{BF}^2(\Cbb)$. Thus by the description of interpolating
sequences for such spaces obtained in \cite{Seip92}, the
upper Beurling
density $D^+(\Sigma)<1$. We will prove that under this circumstances there is a
$\varepsilon>0$ such that $\{\Lambda_{(1+2\varepsilon)k}\}$ is
$L^1$-interpolating.

We argue by contradiction. Suppose not, then, for any $n$ the interpolation
constants at level $k$, $C_k$ for $\Lambda_{(1+1/n)k}$ blow up. Thus by the
dual description of $C_k$ given in Proposition~\ref{duality} we can find
sequences of vectors $\{a_\lambda\}_{\lambda\in \Lambda_{(1+1/n)k}}$ such that
$\sup_{\Lambda_{(1+1/n)k}}|a_\lambda|=1$ and 
\[
\sup_{x\in X} k^{-1} \left | \sum_{\Lambda_{(1+1/n)k}} \dotprod{a_\lambda}
{\Pi_k(x, \lambda)}
\right|=o(1),\qquad \text{ as }k\to\infty.
\]
If we fix $n$ and by passing to a subsequence in normal coordinates around the
points $\lambda^*_k$ where $|a_\lambda|$ takes its maximum value, we  can
extract
a subsequence of $\Lambda_{(1+1/n)k}$ as $k\to \infty$ in normal coordinates
that scaled appropriately converges weakly to the separated sequence
$\Sigma_n\subset \Cbb$. Moreover, after taking a subsequence again, there are
subsequences $a_\lambda^k\to a_\sigma^n$ for all $\sigma\in \Sigma_n$. We are
going to prove that in this case 
\[
 f_n(z):=\sum_{\sigma\in \Sigma_n} a_\sigma^n e^{\bar \sigma z-1/2
|\sigma|^2}\equiv 0,
\]
with $|a_0|=1$, and $\sup_{\sigma} |a_\sigma^n|\le 1$.

To see this we will prove that for any $\varepsilon>0$, $\sup_{|z|<1}
|f_n(z)|e^{-|z|^2}\le\varepsilon$. 

Observe that since $\Sigma_n$ is separated and $|a_\sigma^n|\le 1$,
the decay of the Bargmann-Fock kernel away from the diagonal implies that for
any $\varepsilon>0$ it is possible to find $R>0$ such that 
\[
\sup_{|z|<1} \left |\sum_{\sigma^n\in \Sigma_n, |\sigma|>R} a_\sigma^n e^{\bar
\sigma z - \frac1{2} |\sigma|^2}\right|e^{- \frac1{2} |z|^2} \le \varepsilon
\]
So we only need to care about the points $\sigma\in \Sigma_n\cap D(0,R)$. But
this we can deal with because, with certain abuse of notation,
\[
k^{-1}
 \sum_{\lambda\in \Lambda_{(1+1/n)k}\cap D(\lambda^*_k, R/\sqrt{k})} 
 \dotprod{a_\lambda}{\Pi_k(x, \lambda)}
\to \frac1{\pi}\sum_{\sigma\in \Sigma_n, |\sigma|<R} a_\sigma^n e^{\bar
\sigma z - \frac1{2} |\sigma|^2 - \frac1{2} |z|^2} 
\]
uniformly in $|z|<1$ when the section is expressed in appropriately scaled
normalized coordinates around $\lambda_k^*$. This property is usually called
the universality of the reproducing kernels and it is proved in
\cite[Theorem 3.1]{BleShiZel00}. Actually in \cite{BleShiZel00} it is
assumed that $X$ is equipped with the metric induced by the curvature of
the line bundle, but since the condition \eqref{density} is invariant under
change of metric we may also assume that this is the case. The sum
\[
k^{-1}
\Bigl| \sum_{\lambda\in \Lambda_{(1+1/n)k}\cap D(\lambda^*_k, R/\sqrt{k})}
\dotprod{a_\lambda}{\Pi_k(x, \lambda)}
\Bigr|\le \varepsilon,
\]
if $k$ is big enough because the global sum for all $\lambda\in
\Lambda_{(1+1/n)k}$ converges to zero and the terms $\lambda$ outside the ball
$D(\lambda^*_k, R/\sqrt{k})$ are small when $R$ is big because
$\Lambda_{(1+1/n)k}$ is separated and there is a fast decay of the normalized
reproducing kernel away from the diagonal \eqref{EstimateOffDiagonal}.

Finally we have proved that $f_n\equiv 0$ and $\{a_\sigma^n\}$ is uniformly
bounded sequence with $a_0=1$. We can take a subsequence as $n\to\infty$ and we
find $\Sigma_n\to \Sigma$ weakly  and there is a bounded sequence
$\{a_\sigma\}$ such that $f(z)=\sum a_\sigma e^{\bar \sigma
z-|\sigma|^2/2}\equiv 0$ and $|a_0|=1$. This is clearly not possible since
$\Sigma\in W(\Lambda)$ and thus it has $D^+(\Lambda)<1$, thus $\Lambda$ is
interpolating for the $L^1$ Bargmann-Fock space and this means that by duality
\[
 \sup_\sigma |a_\sigma| \le C \sup_{z\in\Cbb} |\sum a_\sigma e^{\bar\sigma
z-|\sigma|^2/2}|e^{-|z|^2}.
\]
We have thus proved that $\{\Lambda_{(1+2\varepsilon)k}\}$ is
$L^1$-interpolation. We finish the proof by observing that
by Proposition~\ref{saltinterp} this implies that 
$\{\Lambda_{(1+\varepsilon)k}\}$ is $L^2$-interpolation.
\end{proof}

\def\cprime{$'$}
\providecommand{\bysame}{\leavevmode\hbox to3em{\hrulefill}\thinspace}
\providecommand{\MR}{\relax\ifhmode\unskip\space\fi MR }
\providecommand{\MRhref}[2]{%
  \href{http://www.ams.org/mathscinet-getitem?mr=#1}{#2}
}
\providecommand{\href}[2]{#2}


\begin{thebibliography}{BBWN11}

\bibitem[AOC12]{AmeOrt12}
Y.~Ameur and J.~Ortega-Cerd{\`a}, \emph{Beurling--{L}andau densities of
  weighted {F}ekete sets and correlation kernel estimates}, J. Funct. Anal.
  \textbf{263} (2012), no.~7, 1825--1861. \MR{2956927}

\bibitem[AFK11]{AscFeiKai11}
G.~Ascensi, H.~Feichtinger and N.~Kaiblinger, \emph{Dilation of the {W}eyl
  symbol and {B}alian-{L}ow theorem}, Trans. Amer. Math. Soc, to appear.

\bibitem[BBS08]{BerBerSjo08}
R.~Berman, B.~Berndtsson, and J.~Sj{\"o}strand, \emph{A direct
  approach to {B}ergman kernel asymptotics for positive line bundles}, Ark.
  Mat. \textbf{46} (2008), no.~2, 197--217. \MR{2430724 (2009k:58050)}  
  
\bibitem[BBWN11]{BerBouWit11}
R.~Berman, S.~Boucksom, and D.~Witt~Nystr{\"o}m, \emph{Fekete
  points and convergence towards equilibrium measures on complex manifolds},
  Acta Math. \textbf{207} (2011), no.~1, 1--27. \MR{2863909}

\bibitem[Ber03]{Berndtsson03}
B.~Berndtsson, \emph{Bergman kernels related to {H}ermitian line bundles over
  compact complex manifolds}, Explorations in complex and {R}iemannian
  geometry, Contemp. Math., vol. 332, Amer. Math. Soc., Providence, RI, 2003,
  pp.~1--17. \MR{2016088 (2004m:58045)}

\bibitem[Ber10]{Berndtsson10}
\bysame, \emph{An introduction to things {$\overline\partial$}}, Analytic and
  algebraic geometry, IAS/Park City Math. Ser., vol.~17, Amer. Math. Soc.,
  Providence, RI, 2010, pp.~7--76. \MR{2743815 (2012c:32052)}

\bibitem[BSZ00]{BleShiZel00}
P.~Bleher, B.~Shiffman, and S.~Zelditch, \emph{Universality and
  scaling of correlations between zeros on complex manifolds}, Invent. Math.
  \textbf{142} (2000), no.~2, 351--395. \MR{1794066 (2002f:32037)}  
  
\bibitem[Bl{\"u}90]{Blumlinger90}
M.~Bl{\"u}mlinger, \emph{Asymptotic distribution and weak convergence on
  compact {R}iemannian manifolds}, Monatsh. Math. \textbf{110} (1990), no.~3-4,
  177--188. \MR{MR1084310 (92h:58033)}

\bibitem[GM13]{GroMal11}
K.~Grochenig and E.~Malinnikova, \emph{Phase space localization of {R}iesz bases
  for {$L^2(\mathbb R^n)$}}, Rev. Mat. Ibero. \textbf{29} (2013), 115--134.

\bibitem[HM11]{HsiMar11}
C.Y.~Hsiao and G.~Marinescu, \emph{Asymptotics of spectral function of lower
  energy forms and {B}ergman kernel of semi-positive and big line bundles},
  arXiv:1112.5464v1, 2011.

\bibitem[Kod54]{Kodaira54}
K.~Kodaira, \emph{On {K}\"ahler varieties of restricted type (an intrinsic
  characterization of algebraic varieties)}, Ann. of Math. (2) \textbf{60}
  (1954), 28--48. \MR{0068871 (16,952b)}

\bibitem[Lan67]{Landau67}
H.~J. Landau, \emph{Necessary density conditions for sampling and interpolation
  of certain entire functions}, Acta Math. \textbf{117} (1967), 37--52.

\bibitem[Lin01]{Lindholm01}
N.~Lindholm, \emph{Sampling in weighted {$L\sp p$} spaces of entire
  functions in {${\mathbb C}\sp n$} and estimates of the {B}ergman kernel}, J.
  Funct. Anal. \textbf{182} (2001), no.~2, 390--426. \MR{2002g:32007}

\bibitem[MM07]{MaMa07}
X.~Ma and G.~Marinescu, \emph{Holomorphic {M}orse inequalities and
  {B}ergman kernels}, Progress in Mathematics, vol. 254, Birkh\"auser Verlag,
  Basel, 2007. \MR{2339952 (2008g:32030)}

\bibitem[Mar07]{Marzo07}
J.~Marzo,
\emph{Marcinkiewicz-Zygmund inequalities and interpolation by spherical harmonics},
J. Funct. Anal. \textbf{250} (2007), no.~2, 559--587.

\bibitem[MOC10]{MarOrt10}
J.~Marzo and J.~Ortega-Cerd{\`a}, \emph{Equidistribution of {F}ekete
  points on the sphere}, Constr. Approx. \textbf{32} (2010), no.~3, 513--521.
  \MR{2726443 (2011j:65057)}

\bibitem[Mo{\u\i}66]{Moishezon66}
B.G.~Mo{\u\i}{\v{s}}ezon, \emph{On {$n$}-dimensional compact complex manifolds
  having {$n$} algebraically independent meromorphic functions.
  {I},{II},{III}}, Izv. Akad. Nauk SSSR Ser. Mat. \textbf{30} (1966), 133--174.
  \MR{0216522 (35 \#7355a)}
  
\bibitem[NO12]{NitOle12}
S.~Nitzan, A.~Olevskii, \emph{Revisiting Landau’s density theorems for
Paley–Wiener spaces}, C. R. Acad. Sci. Paris, Ser. I \textbf{350} (2012),
509--512. \MR{2929058}

\bibitem[OCP12]{OrtPrid12}
J.~Ortega-Cerd{\`a} and B.~Pridhnani, \emph{Beurling--{L}andau's
  density on compact manifolds}, J. Funct. Anal. \textbf{263} (2012), no.~7,
  2102--2140. \MR{2956935}

\bibitem[RS13]{RotSer13}
S.~Rota Nodari and S.~Serfaty, \emph{Renormalized energy equidistribution and 
local charge balance in 2D Coulomb systems}, arXiv:1307.3363  
  
\bibitem[Sei92]{Seip92}
K.~Seip, \emph{Density theorems for sampling and interpolation in the
  {B}argmann-{F}ock space. {I}}, J. Reine Angew. Math. \textbf{429} (1992),
  91--106. \MR{93g:46026a}

\bibitem[SW92]{SeiWal92}
K.~Seip and R.~Wallst{\'e}n, \emph{Density theorems for sampling and
  interpolation in the {B}argmann-{F}ock space. {II}}, J. Reine Angew. Math.
  \textbf{429} (1992), 107--113. \MR{93g:46026b}
  
\bibitem[Siu84]{Siu84}
Y.T.~Siu, \emph{A vanishing theorem for semipositive line bundles over
  non-{K}\"ahler manifolds}, J. Differential Geom. \textbf{19} (1984), no.~2,
  431--452. \MR{755233 (86c:32029)}

\bibitem[Tia90]{Tian90}
G.~Tian, \emph{On a set of polarized {K}\"ahler metrics on algebraic
  manifolds}, J. Differential Geom. \textbf{32} (1990), no.~1, 99--130.
  \MR{1064867 (91j:32031)}

\bibitem[Vil09]{Villani09}
C.~Villani, \emph{Optimal transport, old and new},
Grundlehren der mathematischen Wissenschaften \textbf{338} (2009).

\bibitem[Zel98]{Zelditch98}
S.~Zelditch, \emph{Szeg{\H o} kernels and a theorem of {T}ian}, Internat.
  Math. Res. Notices (1998), no.~6, 317--331. \MR{1616718 (99g:32055)} 
  
\end{thebibliography}
\end{document}